\documentclass[12pt,a4paper, leqno]{article}
\usepackage{amssymb, amscd, amsthm, amsmath, latexsym, amstext, mathrsfs}
\usepackage[all]{xy}
\usepackage{graphicx}

\usepackage[colorlinks, linkcolor=blue, anchorcolor=green, citecolor=blue]{hyperref}

\newtheorem{thm}{Theorem}[section]
\newtheorem{coro}[thm]{Corollary}
\newtheorem{lemma}[thm]{Lemma}
\newtheorem{prop}[thm]{Proposition}

\theoremstyle{remark}
\newtheorem{remark}[thm]{\textbf{Remark}}

\theoremstyle{definition}
\newtheorem{defn}[thm]{Definition}

\numberwithin{equation}{thm}
\def\Bot{\mbox{\Large $\bot$}}

\newcommand{\cO}{\mathcal{O}}

\newcommand{\bH}{\mathbb{H}}

\newcommand{\N}{\mathbb{N}}
\newcommand{\Z}{\mathbb{Z}}
\newcommand{\Q}{\mathbb{Q}}

\newcommand{\dgf}[1]{\langle\, #1\,\rangle}

\newcommand{\fa}{\mathfrak{a}}
\newcommand{\frs}{\mathfrak{s}}
\newcommand{\fp}{\mathfrak{p}}
\newcommand{\fm}{\mathfrak{m}}
\newcommand{\fn}{\mathfrak{n}}

\newcommand{\ff}{\mathfrak{f}}
\newcommand{\fd}{\mathfrak{d}}

\newcommand{\ord}{\mathrm{ord}}

\newcommand{\al}{\alpha}
\newcommand{\veps}{\varepsilon}

\newcommand{\rep}{\rightharpoonup}

\newcommand{\newpara}{\noindent\refstepcounter{thm}{\bf(\thethm)\;}}

\begin{document}

\title{\textbf{On $k$-universal quadratic lattices over unramified dyadic local fields}}
\author{Zilong HE and Yong HU}
\date{}

\maketitle 

\begin{abstract}
Let $k$ be a positive integer and let $F$ be a finite unramified extension of $\mathbb{Q}_2$ with ring of integers $\mathcal{O}_F$.  An integral (resp. classic) quadratic form over $\mathcal{O}_F$  is called $k$-universal (resp. classic $k$-universal) if it represents all integral (resp. classic) quadratic forms of dimension $k$. In this paper, we provide a complete classification of $k$-universal and classic $k$-universal quadratic forms over $\mathcal{O}_F$. The results are stated in terms of the fundamental invariants associated to Jordan splittings of quadratic lattices.
\end{abstract}




\section{Introduction}

Universal integral quadratic forms have been extensively studied since the early 20th century. When the term  was coined by Dickson in \cite{Dickson29TAMS} (see also \cite{Dickson30}), a universal quadratic form is meant to be a quadratic form with integer coefficients that represents all integers. Ross \cite{Ross32} extended this definition by calling a positive (definite) form over $\Z$ universal if it represents all positive integers. Ramanujan \cite{Ramanujan17} classified all quaternary diagonal positive universal quadratic forms over $\Z$. In a series of papers (\cite{Dickson27BAMSpage63}, \cite{Dickson27AJMpage39},  \cite{Dickson29BAMS}, \cite{Dickson29TAMS}, \cite{Ross32}, \cite{Ross33}), Dickson  and Ross further studied general binary, ternary and quaternary forms  in both the positive definite and the indefinite cases.

%


As a natural generalization of the problem, one can consider the representation of arbitrary integral quadratic forms over general number fields and local fields. Among a huge number of references on the topic, let us mention in particular O'Meara's book \cite{OMeara00}, and the papers \cite{OMeara55AJM}, \cite{OMeara58AJM}, \cite{Riehm64AJM} (on the local theory) and \cite{HsiaKitaokaKneser78crelle}, \cite{HsiaShaoXu98crelle} (on the global theory). For any positive integer $k$, B. M. Kim, M.-H. Kim and Raghavan \cite{KimKimRaghavan97} defined the notion of $k$-universal quadratic forms in the positive definite case over $\Z$. Interesting results for positive definite $k$-universal forms (especially for $k=2$ over $\Z$) can be found in \cite{Mordell30}, \cite{Ko37}, \cite{KimKimRaghavan97}, \cite{KimKimOh99ContempMath249}, \cite{Oh00PAMS}, \cite{Kim04ContempMath344}, etc.

For indefinite forms, the local-global principle, when it holds, is generally recognized as a powerful tool that reduces representability problems over number fields to those over local fields. Xu and Zhang have recently considered the local-global principle for the universal property (i.e. the $k=1$ case of $k$-universality) for quadratic forms over general number fields. They determined in \cite{XuZhang22TAMS} all number fields that admit  locally universal but not globally universal integral forms. The work \cite{HeHuXu22} by Xu and the authors continued the investigation and answered  similar questions about $k$-universality for general $k$. In the proofs of these results, a key step turns out to be a complete determination of  $k$-universal forms over non-dyadic local fields and some partial results in the dyadic case. For $k=1$, Beli's work \cite{Beli20}   complements the analysis over dyadic fields in \cite[\S\,2]{XuZhang22TAMS}, and gives necessary and sufficient conditions for an integral quadratic form over a general dyadic field to be universal. His method builds upon the general theory of bases of norm generators (BONGs), which he developed in his thesis \cite{Beli01thesis} (see also \cite{Beli06}, \cite{Beli10TAMS}, \cite{Beli19}). Using only the more standard theory as presented in \cite{OMeara00}, Earnest and Gunawardana also complete a classification of universal forms over the ring $\Z_p$ of $p$-adic integers for any prime $p$ (\cite{EarnestGunawardana21JNT}).

\medskip

In this paper we work over a general unramified dyadic field $F$, i.e. a finite extension of $\Q_2$ in which $2$ remains a prime element. Our goal is to determine all $k$-universal integral quadratic forms over $F$  for an arbitrary $k$. In the statements and the proofs of our results, we only use the classical invariants associated to Jordan splittings of quadratic lattices (see e.g. \cite[\S\,Chap.\;IX]{OMeara00}). The arguments make heavy use of a representability criterion established in \cite{OMeara58AJM}.

 Without the unramifiedness assumption on $F$, all $k$-universal integral quadratic forms will be classified in \cite{HeHu22pre}. As far as we can see, to treat that general case, we need  a comprehensive knowledge of Beli's theory of BONGs, which in our eyes looks more complicated than the classical theory. Moreover, over ramified dyadic fields we have to state the results in terms of invariants associated to BONGs, and unlike the case $k=1$ discussed in \cite{Beli20}, when $k\ge 2$ it seems  rather complicated to translate those results in the language of Jordan splittings.

\medskip

Following \cite{Dickson30} and \cite{Ross32}, a quadratic form is called \emph{classic} if the Gram matrix of its bilinear form has integral coefficients. The notion of \emph{classic $k$-universal forms} can be defined by restricting to classic forms in the usual definition of $k$-universality (cf. Definition\;\ref{UL1.1}). For arbitrary $k$,  we also classify all classic $k$-universal forms over an unramified dyadic local field. The ramified case will be treated in a forthcoming work of the first named author \cite{He22pre}.

\medskip

\noindent {\bf Notation and terminology.} Throughout this paper,  let $F$ be an \emph{unramified} extension of $\Q_2$ with valuation ring $\cO_F$. Let $F^*=F\setminus\{0\}$ and let $\cO_F^*$ be the group of units in $\cO_F$.  Let $v=v_F: F\to \Z\cup\{+\infty\}$ denote the normalized discrete valuation of $F$. For any (nonzero) fractional ideal $\fa$ in $F$, put
 $\ord(\fa)=\min\{v(\al)\,|\,\al\in \fa\}$.

We fix a unit $\Delta=1+4\rho$ with $\rho\in\cO_F^*$ such that $F(\sqrt{\Delta})\cong F[T]/(T^2+T-\rho)$ is the unique quadratic unramified extension of $F$ (cf. \cite[\S\;93, p.251]{OMeara00}).

As in most research papers in the area, in the study of quadratic forms we adopt the geometric language of quadratic spaces and lattices. Unless otherwise stated, we only consider quadratic spaces over $F$ and quadratic lattices over $\cO_F$ and they are all  assumed to be nonsingular. We usually denote by $Q$ the quadratic form associated to a quadratic space or lattice. Standard notation and terminology about quadratic forms are adopted from \cite{OMeara00}. For example, the quadratic space spanned by a lattice $L$ is denoted by $FL$. For a quadratic space $V$, we often write $d(V)\in F^*$ for a representative of the \emph{discriminant} of $V$ (in the sense of \cite[p.87]{OMeara00}). We also put $d_{\pm}(V):=(-1)^{\frac{(\dim V-1)\dim V}{2}}d(V)$ and call it the \emph{signed discriminant} of $V$.  We write $Q(V):=\{Q(u)\,|\,u\in V\}$ for a quadratic space $V$ and similarly  $Q(L):=\{Q(u)\,|\,u\in L\}$ for a lattice $L$.

The norm and the scale of a lattice $L$ are denoted by $\fn(L)$ and $\frs(L)$ respectively (\cite[(82.E)]{OMeara00}). A modular lattice $L$ is called \emph{proper} if $\fn(L)=\frs(L)$. The norm and the scale of the zero lattice are the zero ideal. Thus, for a nonzero fractional ideal $\fa$ and a lattice $L$, when we write $\frs(L)\subseteq\fa$ or $\fn(L)\subseteq \fa$, the case $L=0$ is allowed. But when $\fa\subseteq\frs(L)$ or $\fa\subseteq\fn(L)$, $L$ must be nonzero.

When a lattice $K$ is represented by another lattice $L$ (in the sense of \cite[p.220]{OMeara00}), we write $K\rep L$. Similarly for quadratic spaces.

 For elements $a_1,\cdots, a_n\in F^*$,  let $\dgf{a_1,\cdots, a_n}$ denote the $n$-ary quadratic space or lattice defined by the diagonal quadratic form $a_1x_1^2+\cdots+a_nx_n^2$. This abuse of notation should not cause confusion since it is always clear from the context whether we mean a space or a lattice.

 For any $\xi\in F^*$ and $\al,\,\beta\in \cO_F$, define the lattice $\xi A(\al,\,\beta)$ as in \cite[p.255, \S\,93B]{OMeara00}.

The binary hyperbolic space is denoted by $\bH$, and for any $m\in\N$, the $2m$-dimensional hyperbolic space is denoted by $\bH^m$.

The symbol $\subset$ is used for a strict inclusion.

\medskip

In the language of lattices, some key definitions can be rephrased as follows:
\begin{defn}\label{UL1.1}
Let  $L$ be a lattice (over $\cO_F$). We say $L$ is \emph{integral} (resp. \emph{classic}) if $Q(L)\subseteq \cO_F$ (resp. $\frs(L)\subseteq \cO_F$).

For $k\in\N^*$, we say that $L$ is $k$-\emph{universal} if it is integral and it represents all integral $k$-dimensional lattices. If $L$ is classic and represents all classic $k$-dimensional lattices, then it is called \emph{classic $k$-universal}. We simply say ``universal''  (resp. ``classic universal'') instead of ``1-universal'' (resp. ``classic 1-universal'').

Clearly,  $L$ is universal if and only if $Q(L)=\cO_F$, and $L$ is classic universal if and only if it is universal and classic.
\end{defn}

As we have mentioned before, Beli  determines all universal lattices over general dyadic fields by using the theory of BONGs in \cite{Beli20}. In the last section of that paper,  a translation of his main theorem in terms of Jordan splittings is given without proof.  After some reviews of preliminary facts in Sections\;\ref{sec2} and \ref{sec3}, especially those from O'Meara's paper \cite{OMeara58AJM}, we reprove the classification of universal lattices over the (unramified dyadic) field $F$ in Section\;\ref{sec4}. Our approach is based on O'Meara's work and independent of the theory of BONGs.

For $k\ge 2$, we obtain new  classification results on $k$-universal and classic $k$-universal forms. The next two theorems, concerning the case with $k$ even, are proved in Section\;\ref{sec5}.

\begin{thm}\label{UL1.3}
 Let $L=L_1\bot \cdots\bot L_t$ be a Jordan splitting of an $\cO_F$-lattice. Let $k$ be an even integer $\ge 2$.

  Then $L$ is $k$-universal if and only if $\fn(L_1)=\cO_F$, $\frs(L_1)=2^{-1}\cO_F$  (hence $L_1$ is improper, even dimensional and $d_{\pm}(FL_1)\in F^{*2}\cup \Delta F^{*2}$) and one of the following conditions holds:

\begin{enumerate}
  \item $\dim L_1\ge k+4$.
  \item $\dim L_1=k+2$, and one of the following cases happens:

\begin{enumerate}
  \item $d_{\pm}(FL_1)\in F^{*2}$, and if $k>2$, then $2\cO_F\subseteq \fn(L_2)$.
  \item $d_{\pm}(FL_1)\in \Delta F^{*2}$ and $2\cO_F\subseteq \fn(L_2)$.
\end{enumerate}
  %
  \item $\dim L_1=k$, $\dim L_2\ge 2$, $\frs(L_2)=\fn(L_2)=\cO_F$ and one of the following cases happens:

  \begin{enumerate}
    \item $\dim L_2\ge 3$.
    \item $\dim L_2=2$,  and $\frs(L_3)=\fn(L_3)=2\cO_F$.
    \item $\dim L_2=2$, $d_{\pm}(FL_2)\notin F^{*2}\cup \Delta F^{*2}$,  and $4\cO_F=\fn(L_3)$.
  \end{enumerate}
  \item $\dim L_1=k$, $\dim L_2=1$, $\frs(L_2)=\fn(L_2)=\cO_F$, $\frs(L_3)=\fn(L_3)=2\cO_F$ and one of the following cases happens:

  \begin{enumerate}
   \item $\dim L_3\ge 2$.
    \item $\dim L_3=1$,  and $8\cO_F\subseteq\fn(L_4)$.
  \end{enumerate}
\end{enumerate}
\end{thm}
\begin{proof}
  Combine Propositions\;\ref{UL5.5}, \ref{UL5.7} and \ref{UL5.9}.
\end{proof}

\begin{thm}\label{UL1.4}
With the same notation as in Theorem$\;\ref{UL1.3}$, the lattice $L$ is classic $k$-universal if and only if $\frs(L_1)=\fn(L_1)=\cO_F$ and one of the following conditions holds:

\begin{enumerate}
  \item $\dim L_1\ge k+3$.
  \item $\dim L_1=k+2$,  and $\frs(L_2)=\fn(L_2)=2\cO_F$.
  \item $\dim L_1=k+2$, $d_{\pm}(FL_1)\notin F^{*2}\cup \Delta F^{*2}$, and $\fn(L_2)=4\cO_F$.
  \item $\dim L_1=k+1$, $\dim L_2\ge 2$ and $\frs(L_2)=\fn(L_2)=2\cO_F$.
  \item $\dim L_1=k+1$, $\dim L_2=1$, $\frs(L_2)=\fn(L_2)=2\cO_F$, and $8\cO_F\subseteq \fn(L_3)$.
\end{enumerate}
\end{thm}
\begin{proof}
  Combine Propositions\;\ref{UL5.10} and \ref{UL5.11}.
\end{proof}

We have also classified $k$-universal and classic $k$-universal lattices for all odd $k\ge 3$. The results are stated in Theorems\;\ref{UL6.10} and \ref{UL6.16}.


\section{Values of quadratic spaces}\label{sec2}

In this section, we recall some useful facts about  quadratic spaces over local fields. As stated in the introduction, we shall assume $F$ is a finite unramified extension of $\Q_2$, although the results below may hold with fewer restrictions on $F$ (see Remark\;\ref{UL2.7}).

\newpage


\newpara\label{UL2.1} Let $\fa\subseteq F$ be a fractional ideal or $\fa=0$. For a quadratic space $V$ over $F$, we say $V$ \emph{represents} $\fa$ and we write $\fa\rep V$ if there is a vector $w\in V$ such that $\fa=Q(w)\cO_F$. Here the case $w=0$ is allowed, so that we always have $0\rep V$. If $\fa=\al\cO_F$ with $\al\in F^*$, saying that $\fa\rep V$ is the same as saying that $V$ represents an element $\beta\in F^*$ such that $v(\al)\equiv v(\beta)\pmod{2}$. In particular, for any $i\in\Z$, $2^i\cO_F\rep W$ means $\cO_F^*\cap Q(V)\neq\emptyset$ if $i$ is even, or $2\cO_F^*\cap Q(V)\neq\emptyset$ if $i$ is odd.

\begin{prop}[{\cite[p.858, Proposition\;16]{OMeara58AJM}}]\label{UL2.2}
Let $V$ be a quadratic space  of dimension $\ge 3$ over $F$.

Then for every fractional ideal $\fa$ we have $\fa\rep V$.
\end{prop}

\begin{lemma}\label{UL2.3}
  Let $V$ be a binary quadratic space over $F$.

Then $Q(V)=F$ if and only if $V$  is isotropic.
\end{lemma}
\begin{proof}
Suppose that $V$ is anisotropic. If $1\in Q(V)$, then by \cite[(63:15) (ii)]{OMeara00}, $Q(V)^*:=Q(V)\cap F^*$ is a subgroup of index 2 in $F^*$. So the lemma follows.
\end{proof}

\begin{lemma}\label{UL2.4}
 Let $V$ be a ternary anisotropic quadratic space over $F$.  Let $\pi\in\cO_F$ be a uniformizer.

\begin{enumerate}
  \item If $V=\dgf{1,\,\veps_1,\,\veps_2}$ with $\veps_i\in\cO_F^*$, then $\pi\cO_F^*\subseteq Q(V)$ and $-\veps_1\veps_2\notin Q(V)$.
   \item If $V=\dgf{1,\,\veps_1,\,\pi\veps_2}$   with $\veps_i\in\cO_F^*$, then $\cO_F^*\subseteq Q(V)$ and $-\veps_1\veps_2\pi\notin Q(V)$.
  \item For any $\gamma\in F^*$, we have $\gamma\in Q(V)$ or $\gamma\Delta\in Q(V)$.
\end{enumerate}
\end{lemma}
\begin{proof}
  (1) and (2) follow from the following two facts:

  (i) Any quaternary quadratic space with nontrivial determinant over $F$ is isotropic (\cite[(63:17)]{OMeara00}).

  (ii) For any $a,\,b\in F^*$, the ternary quadratic space $\dgf{1,\,-a,\,-b}$ is isotropic if and only if the quaternary space $\dgf{1,\,-a,\,-b,\,ab}$ is isotropic.

  To see (3), just note that $V\bot\dgf{-\gamma}$ or $V\bot\dgf{-\gamma \Delta}$ must have nontrivial determinant, hence is isotropic.
\end{proof}

\begin{lemma}\label{UL2.5}
  Let $c\in \cO_F^*\setminus (\cO_F^{*2}\cup\Delta\cO_F^{*2})$, and let $\pi\in\cO_F$ be a uniformizer.

 Then there exist $\eta,\,\theta,\,\veps\in \cO_F^*$ such that the space $\dgf{1,\,-c,\,-\theta\pi}$ is isotropic and the spaces $\dgf{1,\,-c,\,-\veps\pi}$, $\dgf{1,\,-c,\,-\eta}$ and $\dgf{\Delta,\,-c\,,\,-\eta}$ are anisotropic.
\end{lemma}
\begin{proof}
  By \cite[p.202, Lemma\;3]{Hsia75Pacific} there exists $\eta\in\cO_F^*$ such that the Hilbert symbol $(c,\,\eta)_F$ is $-1$. The extension $F(\sqrt{c})/F$ is totally ramified.
So  $\theta\pi$ is a norm from $F(\sqrt{c})$ for some $\theta\in\cO_F^*$. Thus, $(c,\,\theta\pi)_F=1$. Taking $\veps=\theta\eta$ we get $(c,\,\veps\pi)_F=-1$.

   As $(-c\eta,\,\Delta)_F=1=(-\eta,\,\eta)_F$, we  have $(-c\eta,\,\eta\Delta)_F=(c,\,\eta)_F=-1$. Thus, the quaternary space $\dgf{1,\,c\eta\,,\,-\eta\Delta\,,\,-c\Delta}$ is anisotropic. It follows that the ternary space $\dgf{1,\,-c\Delta,\,-\eta\Delta}$ is isotropic. By scaling we see that $\dgf{\Delta,\,-c,\,-\eta}$ is anisotropic. This completes the proof.
\end{proof}

\begin{coro}\label{UL2.6}
  If $V$ is a binary space over $F$ such that $d_{\pm}(V)\in \cO_F^*\setminus \Delta \cO_F^{*2}$, then for every fractional ideal $\fa$ we have $\fa\rep V$.
 \end{coro}
\begin{proof}
By scaling we may assume $\fa=\cO_F$. Then we need to show $\cO_F^*\cap Q(V)\neq\emptyset$.   Let $d_{\pm}(V)=c\in \cO_F^*$.  We may assume $V=\dgf{\pi,\,-c\pi}$ for some uniformizer $\pi$. By Lemma\;\ref{UL2.5}, there exists $\theta\in \cO_F^*$ such that $\dgf{1,\,-c,\,-\theta\pi}$ is isotropic. Then $\dgf{\pi,\,-c\pi \,,\,-\theta}$ is isotropic, whence $\theta\in Q(V)$.
\end{proof}

\begin{remark}\label{UL2.7}
The reader may check that in Proposition\;\ref{UL2.2} and Lemmas\;\ref{UL2.3} and \ref{UL2.4}, the field $F$ can be replaced with any non-archimedean local field of characteristic $\neq 2$. In Lemma\;\ref{UL2.5} and Corollary\;\ref{UL2.6}, it suffices to assume that $F$ is a finite (not necessarily unramified) extension of $\Q_2$.
\end{remark}
%
%
%
%
%


\section{A representability criterion}\label{sec3}

We now recall some definitions and results from  O'Meara's paper \cite{OMeara58AJM}, which are key ingredients in the proofs of our main results.

\medskip

\newpara\label{UL3.1}  Let  $L=L_1\bot\cdots\bot L_t$ be a Jordan splitting of a lattice $L$. For each $i\in\Z$, put
\[
\begin{split}
  [\le i]_L:&=\{1\le r\le t\,|\, \ord(\frs(L_r))\le i\}=\{r\,|\, 2^i\cO_F\subseteq \frs(L_r)\}\,,\\
   (i)_L:&=\{1\le r\le t\,|\, \ord(\fn(L_r))\le i\}=\{r\,|\, 2^i\cO_F\subseteq \fn(L_r)\}\,,\\
   [i]_L:&=\{1\le r\le t\,|\, 2^i\cO_F\subseteq \frs(L_r) \text{ or } 2^{i+1}\cO_F=\frs(L_r)\neq \fn(L_r)\}\,.
\end{split}
\]
Define
\[
L_{\le i}:=\underset{{r\in [\le i]_L}}{\Bot}L_r\,,\;\; L_{(i)}:=\underset{{r\in (i)_L}}{\Bot}L_r\,,\;\; L_{[i]}:=\underset{{r\in [i]_L}}{\Bot}L_r\,.
\]

The sublattices $L_{\le i}$, $L_{(i)}$ and $L_{[i]}$ may depend on the choice of the Jordan splitting, but the sets $[\le i]_L$, $(i)_L$ and $[i]_L$ are independent of that choice. (In \cite{OMeara58AJM}, the lattices $L_{\le i},\,L_{(i)}$ and $L_{[i]}$ are denoted by $\mathfrak{L}_i,\,\mathfrak{L}_{(i)}$ and $\mathfrak{L}_{[i]}$ respectively.)

We also define
\[
\begin{split}
\fd_i(L):&=\begin{cases}
 d(L_{\le i})\cO_F= \prod_{r\in [\le i]_L}\frs(L_r)^{\dim L_r}\quad &\text{ if } L_{\le i}\neq 0\,,\\
 0 \quad & \text{ if } L_{\le i}=0\,,\\
\end{cases}\\
\Delta_i(L):&=\begin{cases}
  2^{i+1}\cO_F\quad & \text{ if } L \text{ has a proper } 2^{i+1}\text{-modular component},\\
  2^{i+2}\cO_F\quad & \text{ if } L \text{ has no proper } 2^{i+1}\text{-modular component}\\
  & \text{ but a proper } 2^{i+2}\text{-modular component},\\
 0\quad & \text{ otherwise}.
\end{cases}
\end{split}
\](In \cite{OMeara58AJM}, $\fd_i(L)$ and $\Delta_i(L)$ are denoted by $D_i$ and $\Delta_i$ respectively.)

Since the dimensions $\dim L_r$ and the scales $\frs(L_r)$ for $1\le r\le t$ are independent of the Jordan splitting, the fractional or zero ideals $\fd_i(L)$ and $\Delta_i(L)$ are uniquely determined by the lattice $L$.

\

The following is a reformulation of \cite[p. 858, Definition]{OMeara58AJM}:

\begin{defn}\label{UL3.2}
  Given two lattices $\ell$ and $L$, we say that $\ell$ has a \textbf{\emph{lower type}} than $L$ if the following conditions hold for all $i\in\Z$:

  \begin{enumerate}
     \item $\dim\ell_{\le i}\le \dim L_{\le i}$.
  \item If $\dim \ell_{\le i}=\dim L_{\le i}>0$, then $\ord(\fd_i(\ell)\fd_i(L))$ is even.
  \item If $\dim \ell_{\le i}=\dim L_{\le i}$, then the following conditions hold:
  \begin{enumerate}
    \item $L$ has a proper $2^{i+1}$-modular component if $\ell$ does;
    \item $\ell$ has a proper $2^i$-modular component if $L$ does.
  \end{enumerate}
  \item If $\dim \ell_{\le i}=\dim L_{\le i}-1>0$, then
  \begin{enumerate}
    \item in  case $\ord(\fd_i(\ell)\fd_i(L))\equiv i+1\pmod{2}$, $\ell$ has a proper $2^{i}$-modular component if $L$ does;
    \item in case $\ord(\fd_i(\ell)\fd_i(L))\equiv i\pmod{2}$, $L$ has a proper $2^{i+1}$-modular component if $\ell$ does.
  \end{enumerate}
  \end{enumerate}
\end{defn}

\newpara\label{UL3.3} For quadratic spaces $U$ and $W$ such that there exists a representation $\sigma: U\to W$, the orthogonal complement of $\sigma(U)$ in $W$, as a quadratic space, is uniquely determined by $U$ and $W$ up to isomorphism. We denote it by $W/^\bot U$.

\begin{thm}[{\cite[p.864, Theorem\;3]{OMeara58AJM}}]\label{UL3.4}
  Let $\ell$ and $L$ be lattices with given Jordan splittings.

  Then $\ell\rep L$ if and only if $\ell$ has a lower type than $L$ and the following conditions hold for all $i\in\Z$:

  \begin{enumerate}
      \item $F\ell_{[i]}\rep FL_{(i+2)}$, and the orthogonal complement $FL_{(i+2)}/^\bot F\ell_{[i]}$ represents the ideals $\Delta_i(\ell)$ and $\Delta_i(L)$.
    \item If $FL_{(i+2)}/^\bot F\ell_{[i]}\cong \bH$, then $\Delta_i(\ell)\Delta_i(L)\subseteq \Delta_i(\ell)^2$.
    \item  $F\ell_{\le i}\rep FL_{(i+1)}\bot\dgf{2^i}$ and  $\big(FL_{(i+1)}\bot\dgf{2^i}\big)/^\bot F\ell_{\le i}$ represents $2^i$ or $2^{i}\Delta$.
    \item $F\ell_{[i]}\rep FL_{\le i+1}\bot \dgf{2^i}$ and $\big(FL_{\le i+1}\bot \dgf{2^i}\big)/^\bot F\ell_{[i]}$ represents $2^i$ or $2^i\Delta$.
  \end{enumerate}
\end{thm}

\begin{remark}\label{UL3.5}
   In the statement of Condition (V) in \cite[p.864, Theorem\;3]{OMeara58AJM} there is a typo: ``$\mathfrak{L}_{(i+1)}$'' should read ``$\mathfrak{L}_{i+1}$''. In our notation, that condition is rephrased as Condition (4) of Theorem\;\ref{UL3.4}.
\end{remark}

\section{Classification of universal lattices}\label{sec4}

In this section, let $L=L_1\bot\cdots\bot L_t$ be a Jordan splitting of a nonzero integral lattice. Our goal is to prove Theorem\;\ref{UL1.2}, which gives necessary and sufficient conditions for $L$ to be universal.

\medskip

Note that since $F$ is unramified over $\Q_2$, we have $\fn(L)=\fn(L_1)$.

\begin{lemma}\label{UL4.1}
  Let $\veps\in \cO_F^*$. Then the following conditions are equivalent:

  \begin{enumerate}
    \item The two lattices $\ell=\dgf{\veps}$ and $\ell'=\dgf{2\veps}$ both have a lower type than $L$.
    \item $\fn(L_1)=\cO_F$, and if $\dim L_1=1$, then $\frs(L_2)=2\cO_F$.
  \end{enumerate}
\end{lemma}
\begin{proof}
  For each $i\in\Z$, we have
\[
  \ell_{\le i}=\ell_{[i]}=\begin{cases}
  0 \quad & \text{ if } i<0\,,\\
  \ell \quad & \text{ if } i\ge 0\,,
\end{cases}\quad \text{and}\quad \fd_i(\ell)=\begin{cases}
  0 \quad & \text{ if } i<0\,,\\
  \cO_F \quad & \text{ if } i\ge 0\,.
\end{cases}
\]Likewise, we can determine the lattice $\ell'_{\le i}=\ell'_{[i]}$ and the ideal $\fd_i(\ell')$ for $\ell'$.

Let us first assume $\ell$ and $\ell'$ have a lower type than $L$ and show that condition (2) holds.

To start with, note that condition \eqref{UL3.2} (1) yields $\cO_F\subseteq \frs(L_1)$. Since we have assumed $L$ integral, this is equivalent to

(a) $\frs(L_1)=\cO_F$ or $\fn(L_1)=\cO_F=2\frs(L_1)$.

In the case $i=0$, \eqref{UL3.2} (2) for $\ell$ says that if $\dim  L_{\le 0}=1$, then $\ord(\fd_0(L))$ is even. We already have $\cO_F\subseteq\frs(L_1)$ by (a). So this means that if $\dim L_1=1$ and $\frs(L_2)\subseteq 2\cO_F$, then $\ord(\frs(L_1))$ is even. On the other hand, when $i=1$, \eqref{UL3.2} (2) for $\ell'$ demands that if $\dim  L_{\le 1}=1$, then $\ord(\fd_1(L))$ is odd. That is, if $\dim L_1=1$ and $\frs(L_2)\subseteq 4\cO_F\subseteq 2\cO_F$, then $\ord(\frs(L_1))$ is odd. Combined with the condition we have just observed, this implies

(b) if $\dim L_1=1$, then $2\cO_F\subseteq\frs(L_2)$.

By considering the case $i=-1$ in \eqref{UL3.2} (3) for $\ell$, we find that if $\dim  L_{\le -1}=0$ (which together with condition (a) above is equivalent to $\frs(L_1)=\cO_F$), then $L$ has a proper unimodular component (which must be $L_1$). In other words, we have

(c) if $\frs(L_1)=\cO_F$, then $\fn(L_1)=\cO_F$.

Combining (a), (b) and (c), we see that $\fn(L_1)=\cO_F$ and that if $\dim L_1=1$ (so that $\frs(L_1)=\fn(L_1)$), then $\frs(L_2)=2\cO_F$. This proves (1)$\Rightarrow$(2).

Conversely, let us now suppose condition (2) holds. We want to show that both $\ell$ and $\ell'$ have a lower type than $L$.

Since $\cO_F=\fn(L_1)\subseteq \frs(L_1)$, we have $\dim  L_{\le 1}\ge \dim  L_{\le 0}\ge 1$. Hence \eqref{UL3.2} (1) holds for both $\ell$ and $\ell'$.

If $\dim\ell_{\le i}=\dim L_{\le i}>0$, we must have $i\ge 0$ and $\dim L_{\le i}=1=\dim L_1$. Then $\frs(L_1)=\fn(L_1)=\cO_F$ and $\ord(\fd_i(\ell)\fd_i(L))=\ord(\frs(L_1))=0$. So \eqref{UL3.2} (2) is true for $\ell$. If $\dim \ell'_{\le i}>0$, we have $i\ge 1$, $\dim \ell'_{\le i}=1$ and $\dim L_{\le i}\ge \dim L_1+\dim L_2\ge 2$. Therefore, there is no need to check \eqref{UL3.2} (2) for $\ell'$.

If $i<-1$, condition \eqref{UL3.2} (3) holds automatically for $\ell$ and $\ell'$. If $i=-1$, for $\ell'$ that condition is trivial and for $\ell$ it requires that if $\dim  L_{\le -1}=0$, then $L$ has  a proper unimodular component, This statement holds since under the assumption $\fn(L_1)=\cO_F$, the equality $\dim  L_{\le -1}=0$ means $\frs(L_1)=\fn(L_1)=\cO_F$. If $i=0$, then \eqref{UL3.2} (3) is trivially verified for $\ell$, and there is no need to check it for $\ell'$ since $\dim  L_{\le 0}\ge 1>0=\dim \ell'_{\le 0}$. Now suppose $i>0$. Then we don't need to check \eqref{UL3.2} (3) for $\ell'$ since $\dim L_{\le i}\ge 2>1=\dim \ell'_{\le i}$. For $\ell$ the condition means that if $\dim L_{\le i}=1$, then $L$ has no proper $2^i$-modular component. This is true because $\dim L_{\le i}=1$ implies $L_{\le i}=L_1= L_{\le 0}$.

If $i\le 0$, \eqref{UL3.2} (4) holds trivially for $\ell$ and $\ell'$.

 If $i>0$, that condition for $\ell$ requires that if $\dim L_{\le i}=2$ and $\ord(\fd_i(L))\equiv i+1\pmod{2}$, then $L$ has no proper $2^i$-modular component. To check this we may assume $\dim  L_{\le 0}\le 1$. Then the equality $\dim L_{\le i}=2$ holds if and only if $\dim  L_{\le 0}=\dim L_1=1$ and $L$ has a 1-dimensional $2^m$-modular component with $0<m\le i$ but no other modular component with scale $\supseteq 2^i\cO_F$. If $m\neq i$, then $L$ has no $2^i$-modular component. If $m=i$, then $\ord(\fd_i(L))=\ord(\frs(L_1))+i=i\not\equiv i+1\pmod{2}$, noticing that $\frs(L_1)=\fn(L_1)=\cO_F$ since $\dim L_1=1$. Thus we have verified \eqref{UL3.2} (4) for $\ell$.

If $i=1$, \eqref{UL3.2} (4) is trivial for $\ell'$. If $i>1$, that condition for $\ell'$ says that if $\dim L_{\le i}=2$ and $\ord(\fd_i(L))\equiv i\pmod{2}$, then $L$ has no proper $2^i$-modular component. But this is clearly true, since $\dim  L_{\le 1}\ge 1$ under the assumptions in (2).

We have thus finished the proof.
\end{proof}

\begin{lemma}\label{UL4.2}
  The lattice $L$ is universal if and only if all the following conditions hold:

  \begin{enumerate}
    \item $\fn(L_1)=\cO_F$, and if $\dim L_1=1$, then ($t\ge 2$ and) $\frs(L_2)=2\cO_F$.

    \item If $FL_{(1)}\cong\bH$, then $L$ has no proper unimodular component.
    \item $\cO_F^*\subseteq Q(FL_{(2)})\cap Q(FL_{(1)}\bot \dgf{1})$.

    $2\cO_F^*\subseteq Q(FL_{(3)})\cap Q(FL_{(2)}\bot \dgf{2})$.

    \item For all $\veps\in \cO_F^*$, $\big(FL_{(1)}\bot\dgf{1}\big)/^\bot \dgf{\veps}$ represents $1$ or $\Delta$, and  $\big(FL_{(2)}\bot\dgf{2}\big)/^\bot \dgf{2\veps}$ represents $2$ or $2\Delta$.
    \item $\cO_F^*\cap Q(FL_{(0)})\neq\emptyset$ and $2\cO_F^*\cap Q(FL_{(1)})\neq\emptyset$.
    \item For all $\veps\in \cO_F^*$ the following conditions are satisfied:
    \begin{enumerate}
          \item If $L$ has a proper $2$-modular component, then
    $2\cO_F^*\cap Q(FL_{(2)}/^\bot\dgf{\veps})\neq\emptyset$.
      \item If $L$ has a proper $4$-modular component, then
   $\cO_F^*\cap Q(FL_{(3)}/^\bot\dgf{2\veps})\neq\emptyset$.
      \item If $L$ has both a proper $2$-modular component and a proper $4$-modular component, then
      $ \cO_F^*\cap Q(FL_{(3)}/^\bot\dgf{\veps})\neq\emptyset$.
      \item If $L$ has no proper $2$-modular component but has a proper $4$-modular component, then
     $\cO_F^*\cap Q(FL_{(2)}/^\bot\dgf{\veps})\neq\emptyset$.

      \item If $L$ has both a proper $4$-modular component and a proper $8$-modular component, then
      $2\cO_F^*\cap Q(FL_{(4)}/^\bot\dgf{2\veps})\neq\emptyset$.
      \item If $L$ has no proper $4$-modular but has a proper $8$-modular component, then
      $2\cO_F^*\cap Q(FL_{(3)}/^\bot\dgf{2\veps})\neq\emptyset$.
    \end{enumerate}
  \end{enumerate}
\end{lemma}
\begin{proof}
Since having a lower type than $L$ is a necessary condition for being represented by $L$ (Theorem\;\ref{UL3.4}), by Lemma\;\ref{UL4.1}, we may assume condition (1) holds. We need to show that $L$ is universal if and only if it also satisfies conditions (2)--(6). It is easy to see that $L$ is universal if and only if it represents the lattices $\dgf{\veps}$ and $\dgf{2\veps}$ for all $\veps\in \cO_F^*$. So it suffices to translate conditions (1)--(4) in Theorem\;\ref{UL3.4} for the lattices $\ell=\dgf{\veps}$ and $\ell'=\dgf{2\veps}$.

Note that
\[
\Delta_i(\ell)=\begin{cases}
  0 \quad & \text{ if } i\notin \{-2,\,-1\}\,,\\
  \cO_F \quad & \text{ if } i\in \{-2,\,-1\}\,,
\end{cases}\quad \text{and}\quad \Delta_i(\ell')=\begin{cases}
  0 \quad & \text{ if } i\notin \{-1,\,0\}\,,\\
  2\cO_F \quad & \text{ if } i\in \{-1,\,0\}\,.
\end{cases}
\]Thus, condition (2) comes from \eqref{UL3.4} (2) for $\ell'=\dgf{2\veps}$ and $i=-1$. Conditions (3) and (4) follow from \eqref{UL3.4} (1) and \eqref{UL3.4} (3) in the case $i=0$ for $\ell$, and \eqref{UL3.4} (1) and \eqref{UL3.4} (4) in the case $i=1$ for $\ell'$. Condition (5) is derived from the case $i=-2$  for $\ell$ and the case $i=-1$ for $\ell'$ in \eqref{UL3.4} (1). Similarly, condition (6) is obtained from several special cases of \eqref{UL3.4} (1).

Careful verifications show that when conditions (1)--(4) hold, all the conditions in Theorem\;\ref{UL3.4} are satisfied by $\ell$ and $\ell'$. This proves the lemma.
\end{proof}

\begin{prop}\label{UL4.3}
  If  $\dim L_1\ge 4$ and $\fn(L_1)=\cO_F$, then $L$ is universal.
\end{prop}
\begin{proof}
 Clearly, conditions (1) and (2) of Lemma\;\ref{UL4.2} are satisfied. To see that the other conditions of that lemma also hold, notice that $L_1\subseteq L_{(0)}$. Then use Proposition\;\ref{UL2.2} and the fact that  $Q(V)=F$ for every nonsingular quadratic space $V$ of dimension $\ge 4$.
\end{proof}

\begin{prop}\label{UL4.4}
  Suppose $\dim L_1=3$ and  $\fn(L_1)=\cO_F$. Then the following are equivalent:

  \begin{enumerate}
    \item $L$ is universal.
    \item Either $FL_1$ is isotropic or $L_1\subset L_{(2)}$.
    \item Either $FL_1$ is isotropic or $4\cO_F\subseteq \fn(L_2)$.
  \end{enumerate}
\end{prop}
\begin{proof}
  We first prove that when $\dim L_1=3$ and $\fn(L_1)=\cO_F$, $L$ satisfies all the conditions of Lemma\;\ref{UL4.2} except possibly for the inclusion $\cO_F^*\subseteq Q( FL_{(2)})$ in \eqref{UL4.2} (3).

  In fact, since $\dim L_1$ is odd, we have $\frs(L_1)=\fn(L_1)=\cO_F$ and hence $L_1=\dgf{\veps_1,\,\veps_2,\,\veps_3}$ for some $\veps_i\in\cO_F^*$. Clearly, conditions (1) and (2) of Lemma\;\ref{UL4.2} are satisfied.

  Since $ FL_{(1)}\bot\dgf{1}$ and $ FL_{(2)}\bot \dgf{2}$ both have dimension $\ge \dim L_1+1=4$, we have
  \[
  \cO_F^*\subseteq Q( FL_{(1)}\bot \dgf{1})\quad\text{and}\quad 2\cO_F^*\subseteq Q( FL_{(2)}\bot\dgf{2})\,.
  \] Also, note that $2\cO_F^*\subseteq Q(FL_1)\subseteq Q( FL_{(3)})$ by Lemma\;\ref{UL2.4} (1). Condition \eqref{UL4.2} (4) is guaranteed by Lemma\;\ref{UL2.4} (3), and \eqref{UL4.2} (5) and (6) can be verified by using Proposition\;\ref{UL2.2}. For example, to check \eqref{UL4.2} (6.b) it suffices to observe that if $L$ has a proper 2-modular component, then $\dim  FL_{(2)}\ge \dim L_1+1=4$ and hence $\dim \big( FL_{(2)}/^\bot\dgf{\veps}\big)\ge 3$.

  It remains to show that $\cO_F^*\subseteq Q( FL_{(2)})$ holds if and only if $FL_1$ is isotropic or $L_1\subset L_{(2)}$. This again follows easily from Lemma\;\ref{UL2.4} (1).
\end{proof}

\begin{prop}\label{UL4.5}
  Suppose $\dim L_1=2$, $\fn(L_1)=\cO_F$ and $FL_1$ is isotropic.

  Then the following are equivalent:

  \begin{enumerate}
    \item $L$ is universal.
    \item Either  $L_1\cong 2^{-1}A(0,\,0)$ or $L_1\subset L_{(1)}$.
    \item Either $L_1\cong 2^{-1}A(0,\,0)$, or $L_1$ is proper and $\frs(L_2)=\fn(L_2)=2\cO_F$.
  \end{enumerate}
\end{prop}
\begin{proof}
  One can easily check that (2) and (3) are equivalent. We first show that when $\dim L_1=2$, $\fn(L_1)=\cO_F$ and $FL_1\cong\bH$, conditions (1) and (3)--(6) of Lemma\;\ref{UL4.2}. are all satisfied. This is obvious for \eqref{UL4.2} (1). Since $\bH=FL_1\subseteq  FL_{(1)}$, we see easily that conditions \eqref{UL4.2} (3)--(5) hold. If $L$ has a proper 2-modular component $K$, then $ FL_{(2)}\supseteq FL_1\bot FK\cong\bH\bot FK$ and $2\cO_F^*\cap Q(FK)\neq\emptyset$. From this condition \eqref{UL4.2} (6.a) follows immediately. If $L$ has a proper 4-modular component $M$, then $ FL_{(3)}\supseteq FL_1\bot FM=\bH\bot FM$ and $\cO_F^*\cap Q(FM)\neq\emptyset$. So \eqref{UL4.2} (6.b) holds. By similar arguments, we can check  (6.b)--(6.f) of Lemma\;\ref{UL4.2}.

  It is now sufficient to observe that our condition (2) here is equivalent to \eqref{UL4.2} (2), since when $L_1$ is improper with $\fn(L_1)=\cO_F$ and $FL_1$ isotropic, we must have $L_1\cong 2^{-1}A(0,\,0)$.
\end{proof}

\begin{prop}\label{UL4.6}
  Suppose $\dim L_1=2$, $\frs(L_1)=\fn(L_1)=\cO_F$ and $FL_1$ is anisotropic.

  Then  $L$ is universal if and only if $\frs(L_2)=\fn(L_2)=2\cO_F$ and one of the following holds:

\begin{enumerate}
  \item $\dim L_2\ge 2$.
  \item $FL_1\bot FL_2$ is isotropic.
  \item  $8\cO_F\subseteq \fn(L_3)$.
\end{enumerate}
\end{prop}
\begin{proof}
  The assumptions on $L_1$ guarantee conditions (1) and (2) of Lemma\;\ref{UL4.2}. Moreover, we have $L_1=\dgf{\veps_1,\,\veps_2}$ for some $\veps_1,\,\veps_2\in \cO_F^*$. Since $FL_1$ is anisotropic, we have $-\veps_1\veps_2\notin F^{*2}$.

Now we claim that a necessary condition for $L$ to be universal is $L_1\subset L_{(1)}$, i.e. $\frs(L_2)=\fn(L_2)=2\cO_F$. Indeed, if $L_1=L_{(1)}$, then  by  Lemma\;\ref{UL2.4} (1), the relation $\cO_F^*\subseteq Q( FL_{(1)}\bot \dgf{1})$ in \eqref{UL4.2} (3) forces $FL_{(1)}\bot\dgf{1}$ to be isotropic, whence $FL_{(1)}\cong  \dgf{-1,\,-\veps_1\veps_2}$. If $-\veps_1\veps_2\in \Delta F^{*2}$, then the condition $2\cO_F^*\cap Q(FL_{(1)})\neq\emptyset$ in \eqref{UL4.2} (5) fails, by \cite[(63:15) (i)]{OMeara00}. So we have $-\veps_1\veps_2\notin F^{*2}\cup \Delta F^{*2}$. Thus, by Lemma\;\ref{UL2.5}, there exists $\veps\in \cO_F^*$ such that the binary space $\dgf{-\veps_1\veps_2,\,-\veps}$ represents neither 1 nor $\Delta$. But for this $\veps$ we have $( FL_{(1)}\bot\dgf{1})/^\bot\dgf{\veps}\cong \dgf{-\veps_1\veps_2,\,-\veps}$, whence a contradiction to the first condition in \eqref{UL4.2} (4).  The claim is thus proved.

Now we may assume $\frs(L_2)=\fn(L_2)=2\cO_F$. Thus $L_{(1)}=L_1\bot L_2$ and $\dim ( FL_{(2)}\bot \dgf{2})\ge 4$. The inclusion $2\cO_F^*\subseteq Q( FL_{(2)}\bot \dgf{2})$ is thus guaranteed.

Next we determine when the other two inclusion relations in \eqref{UL4.2} hold, i.e.,
\[
\cO_F^*\subseteq Q( FL_{(2)})\quad\text{and}\quad 2\cO_F^*\subseteq Q( FL_{(3)})\,.
\]
If $\dim L_2\ge 2$, then $\dim  FL_{(1)}\ge 4$ and these relations clearly hold. So we shall assume $\dim L_2=1$. Then $L_2=\dgf{2\veps_3}$ for some $\veps_3\in\cO_F^*$. By Lemma\;\ref{UL2.4} (2), we have $\cO_F^*\subseteq Q( FL_{(1)})\subseteq Q( FL_{(2)})$. If $FL_1\bot FL_2=FL_{(1)}$ is isotropic, then certainly $2\cO_F^*\subseteq Q( FL_{(3)})$. If $FL_1\bot FL_2=\dgf{\veps_1,\,\veps_2,\,\veps_3}$ is anisotropic, then $ FL_{(1)}$ does not represent all the elements in $2\cO_F^*$ by Lemma\;\ref{UL2.4} (2). In this case, $2\cO_F^*\subseteq Q( FL_{(3)})$ holds if and only if $L_1\bot L_2\subset L_{(3)}$, that is, $8\cO_F\subseteq L_3$.

It remains to show that when  $\frs(L_2)=\fn(L_2)=2\cO_F$ and one of the three cases stated in the lemma occurs, conditions (4)--(6) of Lemma\;\ref{UL4.2} are satisfied. In fact, the condition $\cO_F^*\cap Q( FL_{(0)})\neq\emptyset$ in \eqref{UL4.2} (5) holds because in the present situation we have $ FL_{(0)}=FL_1=\dgf{\veps_1,\,\veps_2}$. All the other conditions can be easily deduced by using Lemma\;\ref{UL2.4} (3).
\end{proof}

\begin{prop}\label{UL4.7}
  Suppose $\dim L_1=2$, $\frs(L_1)=2^{-1}\cO_F=2^{-1}\fn(L_1)$ and $FL_1$ is anisotropic.

  Then  $L$ is universal if and only if one of the following holds:

  \begin{enumerate}
    \item $\dim L_2\ge 2$ and $2\cO_F\subseteq \fn(L_2)$.
    \item $\dim L_2=1$, $\frs(L_2)=\fn(L_2)=\cO_F$.
    \item $\dim L_2=1$, $\frs(L_2)=\fn(L_2)=2\cO_F$, and $8\cO_F\subseteq \fn(L_3)$.
    \end{enumerate}
\end{prop}
\begin{proof}
  It is clear that conditions (1) and (2) of Lemma\;\ref{UL4.2} are satisfied. Moreover, we have $L_1\cong 2^{-1}A(2,\,2\rho)$ by \cite[(93:11)]{OMeara00}. Therefore, $FL_1\cong \dgf{1,\,-\Delta}$. In particular, $\cO_F^*\cap Q( FL_{(0)})\neq\emptyset$.

  If $L_{(1)}=L_1$, then $2\cO_F^*\cap Q( FL_{(1)})=\emptyset$ by \cite[(63:15) (i)]{OMeara00}. This shows that $L_1\subset L_{(1)}$ is a necessary condition for $L$ to be universal. In the rest of the proof we assume this holds, that is, $2\cO_F\subseteq \fn(L_2)$. In particular, $\dim L_{(1)}\ge\dim L_1+\dim L_2\ge 3$.

  By Lemma\;\ref{UL2.4} (3), we have $2\cO_F^*\cap Q( FL_{(1)})\neq\emptyset$, so that \eqref{UL4.2} (5) holds. Similarly, \eqref{UL4.2} (4) is satisfied. As the spaces $ FL_{(1)}\bot \dgf{1}$ and $ FL_{(2)}\bot \dgf{2}$ have dimension $\ge 4$, we have
  \[
  \cO_F^*\subseteq Q( FL_{(1)}\bot \dgf{1})\quad\text{and}\quad 2\cO_F^*\subseteq Q( FL_{(2)}\bot\dgf{2})\,.
  \]
  If $\dim L_2\ge 2$, then $\dim  FL_{(2)}\ge 4$ and it is easily seen that all the other conditions in Lemma\;\ref{UL4.2} are satisfied. So $L$ is universal in this case.

  Next we consider the case with $\dim L_2=1$.  Since $\frs(L_1)\subseteq \frs(L_1)=2^{-1}\cO_F$ and $2\cO_F\subseteq\fn(L_2)$, we have either $L_2=\dgf{\veps_2}$ or $L_2=\dgf{2\veps_2}$ for some $\veps_2\in\cO_F^*$.

  First assume $L_2=\dgf{\veps_2}$. In this case $FL_1\bot FL_2=\dgf{1,\,-\Delta,\,\veps_2}$ is isotropic by \cite[(63:15) (i)]{OMeara00}, and hence \eqref{UL4.2} (3) holds. Condition \eqref{UL4.2} (6) follows easily from Lemma\;\ref{UL2.4} (3).

  Finally, assume $L_2=\dgf{2\veps_2}$. In this case $L_{(1)}=L_1\bot L_2$ and $ FL_{(1)}=\dgf{1,\,-\Delta,\,2\veps_2}$. So we have $\cO_F^*\subseteq Q( FL_{(1)})\subseteq Q( FL_{(2)})$. Applying Lemma\;\ref{UL2.4} (3) we find easily that conditions (6.b)--(6.f) of Lemma\;\ref{UL4.2} hold. For any $\veps\in \cO_F^*$, the space $W:= FL_{(1)}/^\bot\dgf{\veps}$ has discriminant $d(W)\in 2\cO_F^*F^{*2}$. So by \cite[(63:11)]{OMeara00}, we have $2\in Q(W)$ or $2\Delta\in Q(W)$. This shows that \eqref{UL4.2} (6.a) holds. It remains to check the condition $2\cO_F^*\subseteq Q( FL_{(3)})$ in \eqref{UL4.2} (3). Note that $FL_1\bot FL_2=\dgf{1,\,-\Delta,\,2\veps_2}$ is anisotropic. So the desired inclusion holds if and only if $L_{(3)}$ is strictly larger than $L_1\bot L_2$, that is, $8\cO_F\subseteq \fn(L_3)$. This completes the proof.
  \end{proof}

\begin{prop}\label{UL4.8}
  Suppose $\dim L_1=1$ and $\fn(L_1)=\cO_F$.

  Then $L$ is universal if and only if $\frs(L_2)=\fn(L_2)=2\cO_F$, and one of the the following conditions holds:

  \begin{enumerate}
    \item $\dim L_2\ge 3$.
    \item $\dim L_2=2$,  and $\frs(L_3)=\fn(L_3)=4\cO_F$.
    \item $\dim L_2=1$,  $\dim L_3\ge 2$ and $\frs(L_3)=\fn(L_3)=4\cO_F$.
    \item $\dim L_2=\dim L_3=1$, $\frs(L_3)=\fn(L_3)=4\cO_F$, and $\frs(L_4)=\fn(L_4)=8\cO_F$.
  \end{enumerate}
\end{prop}
\begin{proof}
 We have $L_1=\dgf{\veps}$ for some $\veps_1\in\cO_F^*$. So obviously $\cO_F^*\cap Q( FL_{(0)})\neq\emptyset$. By \eqref{UL4.2} (1) we may assume $\frs(L_2)=2\cO_F$.

 We claim that $ FL_{(1)}\not\cong\bH$ and that if $L$ is universal, then $\dim L_{(1)}\ge 2$ (or equivalently, $\fn(L_2)=2\cO_F=\frs(L_2)$). Indeed, if $\dim L_{(1)}=2$, then we must have $L_2=\dgf{2\veps_2}$ for some $\veps_2\in \cO_F^*$. But then $ FL_{(1)}=\dgf{\veps_1,\,2\veps_2}$ is anisotropic, whence $ FL_{(1)}\not\cong\bH$. If $\dim L_{(1)}=1$, then $L_{(1)}=L_1=\dgf{\veps_1}$. Choosing an element $c\in \cO_F^*\setminus (\cO_F^*\cup \Delta\cO_F^{*2})$ (which is possible since $|\cO_F^*/\cO_F^{*2}|>2$, by \cite[(63:9)]{OMeara00}), either
 \[
 c\veps_1\notin Q(FL_1\bot\dgf{1})=Q( FL_{(1)}\bot\dgf{1})=Q(\dgf{\veps_1,\,1})
 \]or the space $( FL_{(1)}\bot\dgf{1})/^\bot\dgf{c\veps_1}$ is isomorphic to $\dgf{c}$ and hence represents neither 1 nor $\Delta$. Thus \eqref{UL4.2} (4) fails and $L$ cannot be universal in this case. Our claim is thus proved.

 In the rest of the proof we assume $\fn(L_2)=\frs(L_2)=2\cO_F$, so that $L_{(1)}=L_1\bot L_2$ and $L_2=\dgf{2\veps_2,\,\cdots,\,2\veps_r}$ for some $\veps_2,\cdots,\veps_r\in\cO_F^*$. In particular, $2\cO_F^*\cap Q( FL_{(1)})\neq\emptyset$ and therefore \eqref{UL4.2} (5) holds. By our claim above, we don't need to check conditions (1) and (2) of Lemma\;\ref{UL4.2}. Since $ FL_{(1)}\bot\dgf{1}$ contains $\dgf{\veps_1,\,2\veps_2,\,1}$, we have
 $\cO_F^*\subseteq Q( FL_{(1)}\bot\dgf{1})$ by Lemma\;\ref{UL2.4} (2). Similarly, $2\cO_F^*\subseteq Q( FL_{(2)}\bot\dgf{2})$.

 For any $\veps\in \cO_F^*$, the space $\dgf{\veps_1,\,2\veps_2,\,1}/^\bot\dgf{\veps}$ has determinant $2\cO_F^*F^{*2}$, so it represents 1 or $\Delta$ by \cite[(63:11)]{OMeara00}. Similarly, the other condition in \eqref{UL4.2} (4) is satisfied.

 If $\dim L_2\ge 3$, then $\dim  FL_{(2)}\ge 4$, so that conditions (3), (5) and (6) of Lemma\;\ref{UL4.2} can be easily deduced from the universality of quaternary spaces or Lemma\;\ref{UL2.4} (3).

 Now suppose $\dim L_2=2$. Then $L_{(1)}=L_1\bot L_2=\dgf{\veps_1,\,2\veps_2,\,2\veps_3}$. So
 \[
 2\cO_F^*\subseteq Q(\dgf{\veps_1,\,2\veps_2,\,2\veps_3})=Q( FL_{(1)})\subseteq Q( FL_{(3)})\,.
 \]Moreover, by using Lemma\;\ref{UL2.4} (3) we can check that conditions (6.b)--(6.f) of Lemma\;\ref{UL4.2} hold. It remains to check
 \begin{equation}\label{UL4.8.1}
   \cO_F^*\subseteq Q( FL_{(2)})
 \end{equation}
 and
 \begin{equation}\label{UL4.8.2}
   \forall\;\veps\in \cO_F^*\;,\quad 2\cO_F^*\cap Q( FL_{(2)}/^\bot\dgf{\veps})\neq\emptyset\;.
 \end{equation}
 Suppose that $L_{(2)}=L_{(1)}$ and \eqref{UL4.8.1} holds. Then
 \[
 2\cO_F^*\subseteq Q(\dgf{2\veps_1,\,4\veps_2,\,4\veps_3})=Q(\dgf{\veps_2,\,\veps_3,\,2\veps_1})\,.
 \]By Lemma\;\ref{UL2.4} (2), this can happen only if $ FL_{(2)}=\dgf{\veps_1,\,2\veps_2,\,2\veps_3}$ is isotropic. In that case $ FL_{(2)}\cong\bH\bot\dgf{-\veps_1\veps_2\veps_3}$. But then taking $\veps=-\veps_1\veps_2\veps_3\Delta$ we find that $W:= FL_{(2)}/^\bot\dgf{\veps}$ is isomorphic to $\dgf{-\veps,\,-\veps_1\veps_2\veps_3}=\dgf{\veps_1\veps_2\veps_3\Delta,\,-\veps_1\veps_2\veps_3}$. Hence $2\cO_F^*\cap Q(W)=\emptyset$ by \cite[(63:15) (i)]{OMeara00}. This shows that \eqref{UL4.8.1} and \eqref{UL4.8.2} cannot be satisfied simultaneously if $L_{(1)}=L_{(2)}$ (and $\dim L_2=2$). So we must have $L_{(1)}\subset L_{(2)}$, i.e. $\frs(L_3)=\fn(L_3)=4\cO_F$. Under this assumption, we have $\dim  FL_{(2)}\ge 4$, so \eqref{UL4.8.1} and \eqref{UL4.8.2} are both true.

 Finally, let us assume $\dim L_2=1$. Then $L_{(1)}=L_1\bot L_2=\dgf{\veps_1,\,2\veps_2}$. By \cite[(63:11)]{OMeara00}, $ FL_{(1)}$ cannot represent both 1 and $\Delta$. So $L_{(1)}\subset L_{(2)}$ is a necessary condition for \eqref{UL4.8.1}. Thus, we will assume further that $\frs(L_3)=\fn(L_3)=4\cO_F$.

 If $\dim L_3\ge 2$, then $\dim  FL_{(2)}\ge 4$ and we are done.

 Let us suppose $\dim L_3=1$, so that $L_3=\dgf{4\veps_3}$ and $L_{(2)}=L_1\bot L_2\bot L_3=\dgf{\veps_1,\,2\veps_2,\,4\veps_3}$. Thus, $\cO_F^*\subseteq Q( FL_{(2)})$ by Lemma\;\ref{UL2.4} (2).

 We claim that if $L$ is universal, then $L_{(2)}\subset L_{(3)}$. To see this, suppose $L_{(2)}=L_{(3)}$. Then by Lemma\;\ref{UL2.4} (2), the inclusion
 $2\cO_F^*\subseteq Q( FL_{(3)})=Q(\dgf{\veps_1,\,2\veps_2,\,\veps_3})$ (required in \eqref{UL4.2} (3)) holds only if $ FL_{(3)}$ is isotropic, that is, if $ FL_{(3)}\cong\bH\bot\dgf{-2\veps_1\veps_2\veps_3}$. Taking $\veps=-\veps_1\veps_2\veps_3\Delta$, we find that $W:= FL_{(2)}/^\bot\dgf{2\veps}$ is isomorphic to
 $\dgf{-2\veps,\,-2\veps_1\veps_2\veps_3}=\dgf{2\veps_1\veps_2\veps_3\Delta\,,\,-2\veps_1\veps_2\veps_3}$. By \cite[(63:15) (i)]{OMeara00} we have $\cO_F^*\cap Q(W)=\emptyset$, so that \eqref{UL4.2} (6.b) fails and $L$ cannot be universal. This proves our claim.

 We may thus assume $L_{(2)}\subset L_{(3)}$, that is, $\fn(L_4)=\frs(L_4)=8\cO_F$. Under this assumption, $\dim  FL_{(3)}\ge 4$, so that
 $2\cO_F^*\subseteq Q( FL_{(3)})$ and condition \eqref{UL4.2} (3) is verified. Condition \eqref{UL4.2} (6.a) holds by \cite[(63:11)]{OMeara00}, \eqref{UL4.2} (6.d) is vacuous here, and all the remaining conditions in Lemma\;\ref{UL4.2} can be verified by using Lemma\;\ref{UL2.4} (3). This completes the proof.
\end{proof}

\begin{thm}\label{UL1.2}
Let $F$ be a finite unramified extension of $\Q_2$ and let  $L=L_1\bot \cdots\bot L_t$ be a Jordan splitting of an $\cO_F$-lattice.

  Then $L$ is universal if and only if $\fn(L_1)=\cO_F$ and one of the following conditions holds:

  \begin{enumerate}
    \item $\dim L_1\ge 4$.
    \item $\dim L_1=3$, and either $FL_1$ is isotropic or  $4\cO_F\subseteq \fn(L_2)$.
    \item $\dim L_1=2$, and one of the following cases happens:
    \begin{enumerate}
      \item $\frs(L_1)=\cO_F$, $\frs(L_2)=\fn(L_2)=2\cO_F$, and one of the following holds:
      \begin{enumerate}
      \item $\dim L_2\ge 2$.
        \item $\dim L_2=1$ and $FL_1\bot FL_2$ is isotropic.
        \item $\dim L_2=1$ and  $8\cO_F\subseteq \fn(L_3)$.
      \end{enumerate}

      \item $\frs(L_1)=2^{-1}\cO_F$, and  one of the following holds:
      \begin{enumerate}
        \item $L_1\cong 2^{-1}A(0,\,0)$, i.e., $FL_1$ is isotropic.
        \item $\dim L_2\ge 2$ and $2\cO_F\subseteq \fn(L_2)$.
        \item $\dim L_2=1$, $\frs(L_2)=\fn(L_2)=\cO_F$.
        \item $\dim L_2=1$, $\frs(L_2)=\fn(L_2)=2\cO_F$,  and $8\cO_F\subseteq \fn(L_3)$.
      \end{enumerate}
      \end{enumerate}

    \item $\dim L_1=1$, $\frs(L_2)=\fn(L_2)=2\cO_F$, and one of the following holds:
    \begin{enumerate}
     \item $\dim L_2\ge 3$.
      \item $\dim L_2=2$, and $\fn(L_3)=\frs(L_3)=4\cO_F$.
      \item $\dim L_2=1$, $\dim L_3\ge 2$ and $\frs(L_3)=\fn(L_3)=4\cO_F$.
      \item $\dim L_2=\dim L_3=1$, $\frs(L_3)=\fn(L_3)=4\cO_F$,  and $\frs(L_4)=\fn(L_4)=8\cO_F$.
    \end{enumerate}
  \end{enumerate}
\end{thm}
\begin{proof}
  Combine Propositions\;\ref{UL4.3}--\ref{UL4.8}.
\end{proof}

Since a lattice is classic universal if and only if it is universal and classic, we easily get the following corollary of Theorem\;\ref{UL1.2}.

\begin{coro}\label{UL4.10}
 Let $F$ be a finite unramified extension of $\Q_2$ and let  $L=L_1\bot \cdots\bot L_t$ be a Jordan splitting of an $\cO_F$-lattice.

  Then $L$ is classic universal if and only if $\frs(L_1)=\fn(L_1)=\cO_F$ and one of the following conditions holds:

  \begin{enumerate}
    \item $\dim L_1\ge 4$.
    \item $\dim L_1=3$, and either $FL_1$ is isotropic or  $4\cO_F\subseteq \fn(L_2)$.
    \item $\dim L_1=2$,  $\frs(L_2)=\fn(L_2)=2\cO_F$, and one of the following cases happens:
    \begin{enumerate}
      \item $\dim L_2\ge 2$.
        \item $\dim L_2=1$ and $FL_1\bot FL_2$ is isotropic.
        \item $\dim L_2=1$ and  $8\cO_F\subseteq \fn(L_3)$.
     \end{enumerate}

    \item $\dim L_1=1$, $\frs(L_2)=\fn(L_2)=2\cO_F$, and one of the following holds:
    \begin{enumerate}
     \item $\dim L_2\ge 3$.
      \item $\dim L_2=2$, and $\fn(L_3)=\frs(L_3)=4\cO_F$.
      \item $\dim L_2=1$, $\dim L_3\ge 2$ and $\frs(L_3)=\fn(L_3)=4\cO_F$.
      \item $\dim L_2=\dim L_3=1$, $\frs(L_3)=\fn(L_3)=4\cO_F$,  and $\frs(L_4)=\fn(L_4)=8\cO_F$.
    \end{enumerate}
  \end{enumerate}
\end{coro}

\begin{remark}
  As we have mentioned in the introduction, Theorem\;\ref{UL1.2} is a special case of \cite[Theorem\;3.1]{Beli20}. We have included a detailed proof because our approach does not need the theory of BONGs and  \cite[Theorem\;3.1]{Beli20} is stated without proof. Let us now explain that our theorem agrees with Beli's for the unramified dyadic field $F$.

  With notation as in Theorem\;\ref{UL1.2}, put $r_i=\ord(\frs(L_i))$, $u_i=\ord(\fn(L^{\frs(L_i)}))$ and $\mathfrak{w}_i=\mathfrak{w}(L^{\frs(L_i)})$  for each $1\le i\le t$.
 Note that we have  $u_i=\ord(\fn(L_i))$ and $\mathfrak{w}_i=2\frs(L_i)$  by \cite[\S\,93G]{OMeara00}.
 To compare Theorem\;\ref{UL1.2} with \cite[Theorem\;3.1]{Beli20}, we may assume that $\fn(L)=\fn(L_1)=\cO_F$.

  Since $\mathfrak{w}_1=2\frs(L_1)\supseteq 2\fn(L_1)=2\cO_F=\fp$, condition (1) of Theorem\;\ref{UL1.2} is the same as condition (1) of \cite[Theorem\;3.1]{Beli20}.

  Suppose that $\dim L_1=3$. Then $L_1$ must be proper by  \cite[(93:15)]{OMeara00}, hence $\mathfrak{w}_1=2\frs(L_1)=2\cO_F=\fp$. The inequality $u_2\le 2e=2$ means precisely that $4\cO_F\subseteq \fn(L_2)$. This shows that condition (2) of Theorem\;\ref{UL1.2} coincides with that of \cite[Theorem\;3.1]{Beli20}.

  Suppose that $\dim L_1=2$ and $\frs(L_1)=2^{-1}\cO_F$. If $u_2=0$, we get $\frs(L_2)=\cO_F=\fn(L_2)$. If $u_2=1$ and $\dim L_2=1$, then $\frs(L_2)=2\cO_F=\fn(L_2)$. The inequality $u_3\le 2e+1$ is equivalent to $8\cO_F\subseteq \fn(L_3)$. So we see that the conditions (ii)--(iv) in Theorem\;\ref{UL1.2} (3.b) all together are equivalent to the conditions (3.1.1) and (3.1.2) in \cite[Theorem\;3.1]{Beli20}. Thus, condition (3.b) of Theorem\;\ref{UL1.2} is equivalent to condition (3.1) of Beli's theorem.

   Recall that by definition,  $\mathfrak{w}(L_1)=2\fm(L_1)+2\frs(L_1)$ and $\fm(L_1)\subseteq \fn(L_1)=\cO_F$. Since $\frs(L_1)\supseteq \fn(L_1)=\cO_F$, it follows easily that the equality $\mathfrak{w}(L_1)=\fp$  is equivalent to $\frs(L_1)=\cO_F$.

    Suppose that $\dim L_1=2$ and $\frs(L_1)=\cO_F$. Then we have $u_2\le 1$ if and only if $\frs(L_2)=\fn(L_2)=2\cO_F$. To show that Theorem\;\ref{UL1.2} (3.a) is equivalent to (3.2) of \cite[Theorem\;3.1]{Beli20}, we claim that the cases (3.2.1) and (3.2.2) can never happen in the unramified case.  Suppose $u_2>1$, i.e. $\fn(L_2)\subset 2\cO_F$. Then $\fn(L_2)\subseteq 4\cO_F$ and $\mathfrak{w}_2=2\frs(L_2)\subseteq 4\cO_F$. Here $r_1=0$. If $u_2$ is even, then $4\fp^{r_1+u_2-2[u_2/2]}=4\cO_F$; so the condition $\mathfrak{w}_2\supset $ does not hold. If $u_2$ is odd, then the condition $\mathfrak{w}_2\supset $ forces $\mathfrak{w}_2=4\cO_F$. But this implies $\frs(L_2)=2\cO_F$ and $\fn(L_2)=4\cO_F$, so that $u_2=2$, which is a contradiction. Therefore, the case (3.2.1) in \cite[Theorem\;3.1]{Beli20} cannot happen.

    Now suppose (3.2.2) of \cite[Theorem\;3.1]{Beli20} holds. Then $r_2=u_2>1$ and $t\ge 3$.  If $u_2+u_3$ is odd, then $\ff_2=\frs(L_2)^{-2}\fn(L_2)\fn(L_3)$ by \cite[\S\,93E, p.264]{OMeara00}. The condition $\ff_2\supset 4\fp^{r_1-2[u_2/2]}$ implies $\ff_2\supset\cO_F$ since $[u_2/2]\ge 1$. But
    \[
    \frs(L_2)^2\supset \frs(L_2)\frs(L_3)\supseteq \fn(L_2)\fn(L_3)\,.
    \]So we obtain $\ff_2=\frs(L_2)^{-2}\fn(L_2)\fn(L_3)\subset \cO_F$, a contradiction.

     Now suppose $u_2+u_3$ is even. Note that $u_3\ge r_3>r_2=u_2>1$. Hence
     \[
         1+r_3-u_2>1+u_2-u_2>2-u_2 \quad\text{and}\quad     \frac{1}{2}(u_3+u_2)+1-u_2>2-u_2\,.
      \]     According to \cite[\S\,93E, p.264]{OMeara00} (and \cite[\S\,93G]{OMeara00}), we have
    \[
    \begin{split}
    \ff_2&=\frs(L_2)^{-2}\big(\fd(2^{u_2+u_3})+2^{u_2}\mathfrak{w}_3+2^{u_3}\mathfrak{w}_2+2\fp^{\frac{1}{2}(u_2+u_3)+r_2}\big)\\
    &=\frs(L_2)^{-2}\big(0+2^{u_2}2\frs(L_3)+2^{u_3}2\frs(L_2)+2\fp^{\frac{1}{2}(u_2+u_3)+r_2}\big)\\
    &=2^{-2r_2}\big(2^{u_2+1+r_3}\cO_F+2^{u_3+1+r_2}\cO_F+2^{\frac{1}{2}(u_2+u_3)+1+r_2}\cO_F\big)\\
     &=2^{-2r_2}\big(2^{u_2+1+r_3}\cO_F+2^{\frac{1}{2}(u_2+u_3)+1+r_2}\cO_F\big)\\
    &=2^{1+r_3-u_2}\cO_F+2^{\frac{1}{2}(u_3+u_2)+1-u_2}\cO_F\subset 2^{2-u_2}\cO_F\,.
    \end{split}
    \]Note that $4\fp^{r_1-2[u_2/2]}\supseteq 2^{3-u_2}\cO_F$. So the condition $\ff_2\supset 4\fp^{r_1-2[u_2/2]}$ does not hold, whence a contradiction.

    Finally, let us show that condition (4) of Theorem\;\ref{UL1.2} is equivalent to that of \cite[Theorem\;3.1]{Beli20}. We may assume $\dim L_1=1$, $\frs(L_1)=\fn(L_1)=\cO_F$ and $\frs(L_2)=\fn(L_2)=2\cO_F$. Then $u_3\le 2e=2$ if and only if $\frs(L_3)=\fn(L_3)=4\cO_F$. Thus, conditions (a) and (b) in Theorem\;\ref{UL1.2} (4) are equivalent to conditions (4.1) and (4.2) in Beli's theorem. To finish the proof, we may further assume $\dim L_2=1$.

  The condition $\mathfrak{w}_3\supset 4\fp^{u_3-2[(u_3-1)/2]}$ means $1+r_3<2+u_3-2[(u_3-1)/2]$. Since $1+r_3\ge u_3\ge r_3\ge 2$, it is easily checked that the above condition holds if and only if $\frs(L_3)=\fn(L_3)=4\cO_F$. This proves that condition (c) in Theorem\;\ref{UL1.2} (4) is equivalent to (4.3.1) in Beli's theorem.

  Now we further assume $\dim L_3=1$ and $t\ge 4$. Then $u_4\ge r_4>u_3=r_3\ge 2$. If $u_3+u_4$ is even, then $1+r_3\ge 4$ and $\frac{1}{2}(u_3+u_4)+1\ge \frac{6}{2}+1=4$.
  But
  \[
  \ff_3=2^{1+r_4-u_3}\cO_F+2^{\frac{1}{2}(u_3+u_4)+1-u_3}\cO_F\quad \text{and}\quad 4\fp^{-2[(u_3-1)/2]}\supseteq 2^{4-u_3}\cO_F\,.
  \]So the condition $\ff_3\supset 4\fp^{-2[(u_3-1)/2]}$ cannot hold in this case.

  If $u_3+u_4$ is odd, we have $\ff_3=2^{u_3+u_4-2r_3}\cO_F=2^{u_4-u_3}\cO_F$. An elementary verification shows that $\ff_3\supset 4\fp^{-2[(u_3-1)/2]}$ if and only if $u_3=2$ and $r_4=u_4=3$. We have thus proved that  condition (d) in Theorem\;\ref{UL1.2} (4) is equivalent to (4.3.2) in Beli's theorem.
\end{remark}

\section{Classification of $k$-universal lattices for $k$ even}\label{sec5}

As in the previous section, let $L=L_1\bot\cdots\bot L_t$ be a Jordan splitting of a nonzero integral lattice. Let $k\ge 2$ be an integer. We prove Theorem\;\ref{UL1.3} in this section.

\medskip

\newpara\label{UL5.1}
  Let $\ell$ be a lattice of dimension $k$. We say that $\ell$ is \textbf{\emph{basic}} if it has a splitting of the form
  \[
  \ell=I_{-1}\bot I_0\bot P_0\bot P_1\bot I_1
  \]where for each $j$, $I_j$ is improper $2^{j}$-modular or 0 and $P_j$ is proper $2^j$-modular or 0.

  Clearly, such a lattice is classic if and only if $I_{-1}=0$.

  If $\ell$ is basic with $I_1=0$, we say that $\ell$ is \textbf{\emph{dominant}}.

  It is easy to see that $L$ is $k$-universal if and only if it represents all dominant $k$-dimensional lattices. If $L$ is classic, it is classic $k$-universal if and only if it represents all classic basic lattices of dimension $k$.   

\begin{prop}\label{UL5.2}
  Let $k\ge 2$ be an even integer.  Then the following statements are equivalent:

  \begin{enumerate}
    \item Every dominant lattice of dimension $k$ has a lower type than $L$.
     \item $\frs(L_1)=2^{-1}\cO_F=2^{-1}\fn(L_1)$,  $\dim L_1\ge k$, and if $\dim L_1=k$, then $2\cO_F\subseteq \frs(L_2)$.
  \end{enumerate}
\end{prop}
\begin{proof}
  Let $\ell=I_{-1}\bot I_0\bot P_0\bot P_1$ be an arbitrary dominant lattice of dimension $k$. We have
  \begin{equation}\label{UL5.2.1}
    \ell_{\le i}=\begin{cases}
    0 \quad & \text{ if } i<-1\,,\\
    I_{-1} \quad & \text{ if } i=-1\,,\\
    I_{-1}\bot I_0\bot P_0 \quad & \text{ if } i=0\,,\\
    \ell\quad & \text{ if } i\ge 1\,.\\
  \end{cases}
  \end{equation}
  Since $I_{-1}$ and $I_0$ are even dimensional and $\frs(P_0)=\cO_F$ when $P_0\neq 0$, we have
  \begin{equation}\label{UL5.2.2}
    \ord(\fd_i(\ell))\equiv \dim P_1\pmod{2}\quad\text{ if } \ell_{\le i}=\ell\,.
  \end{equation}

We need to determine when all the conditions in Definition\;\ref{UL3.2} hold for all $\ell$.

First, we claim that both conditions \eqref{UL3.2} (1) and \eqref{UL3.2} (2) hold for all $\ell$ if and only if condition (2) in this proposition holds.

To prove the claim, note that condition \eqref{UL3.2} (1) (if required for all $\ell$) is equivalent to $\dim  L_{\le -1}\ge k$. Let us assume this condition holds, i.e., $\frs(L_1)=2^{-1}\cO_F=2^{-1}\fn(L_1)$ (since we have assumed $L$ integral) and $\dim L_1\ge k$.

If $\dim L_{\le i}>0$, then we must have $i\ge -1$ and $\dim L_{\le i}\ge \dim L_1\ge k$. But $\dim \ell_{\le i}\le k$. So the assumption $\dim \ell_{\le i}=\dim L_{\le i}>0$ will imply that $\dim \ell_{\le i}=k$ and that $\ord(\fd_i(L))=\ord(\frs(L_1)^k)\equiv k\equiv 0\pmod{2}$. In view of \eqref{UL5.2.2}, we can translate condition \eqref{UL3.2} (2) into the following:

\emph{If $\dim P_1$ is odd, then the condition $\dim\ell_{\le i}=\dim L_{\le i}>0$ never holds.}

For $i=-1,\,0$, by \eqref{UL5.2.1}, $\dim \ell_{\le i}=k$ implies $P_1=0$. For $i\ge 1$, the above condition means that if $\dim P_1$ is odd, then $\dim  L_{\le 1}>k$. Therefore, if \eqref{UL3.2} (2) holds for all $\ell$, we must have $\dim  L_{\le 1}>k$. Our claim is thus proved.

Now let us assume condition (2) holds. It suffices to show that under this assumption, conditions (3) and (4) in Definition\;\ref{UL3.2} are satisfied for all $\ell$.

If $i<-1$, $\ell$ has no proper $2^{i+1}$-modular component and $L$ has no proper $2^i$-modular component. For $i\ge -1$, if $\ell_{\le i}$ and $L_{\le i}$ have the same dimension, then their dimension must be $k$ (which is necessarily $\dim L_1$), and hence $\ell$ has no $2^{i+1}$-modular component and $L$ has no proper $2^i$-modular component. This verifies condition \eqref{UL3.2} (3).

There is no need to check \eqref{UL3.2} (4) if $i<-1$. Since $ L_{\le -1}=L_1$ and  $\ell_{\le -1}=\dim I_{-1}$ are both even dimensional, the equality $\dim  L_{\le -1}=\dim \ell_{\le -1}+1$ never happens. So  we don't need to check \eqref{UL3.2} (4) for $i=-1$ either.

For $i=0$, if $\dim  L_{\le 0}=\dim \ell_{\le 0}+1$, we must have $\dim L_1=\dim  L_{\le -1}=k$, $\dim L_2=1$, $\frs(L_2)=\cO_F$ and $\dim \ell_{\le 0}=\dim\ell=k$. Thus $\ord(\fd_0(L))\equiv 0\pmod{2}$, $P_1=0$, $\ord(\fd_0(\ell))\equiv 0\pmod{2}$ and $\ell$ has no proper $2$-modular component. So we see that \eqref{UL3.2} (4) holds for $i=0$.

Now suppose $i\ge 1$. Then $\ell$ has no $2^{i+1}$-modular component and \eqref{UL3.2} (4) requires the following:

\emph{If $\dim L_{\le i}=\dim \ell_{\le i}+1$, $\ord(\fd_i(\ell)\fd_i(L))\equiv i+1\pmod{2}$ and $L$ has  a proper $2^i$-modular component, then $\ell$ has  a proper $2^i$-modular component.}

Let us assume $\dim L_{\le i}=\dim \ell_{\le i}+1$. Then $\dim L_1=k$, and $L$ has a 1-dimensional $2^m$-modular component with $0\le m\le i$ and no other modular component with scale $\supseteq 2^i\cO_F$. If $m<i$, then $L$ has no $2^i$-modular component and there is nothing to check. So let us suppose $m=i$,  i.e., $L$ has a 1-dimensional $2^i$-modular component but no other modular component with scale $\supseteq 2^i\cO_F$. As we have assumed $\frs(L_2)\supseteq 2\cO_F$ (in our condition (2)), in this case we must have $i=1$. Then $\ord(\fd_1(\ell)\fd_1(L))\equiv 1+1\pmod{2}$ if and only if
$\dim P_1$ is odd, which implies in particular $P_1\neq 0$. Thus condition \eqref{UL3.2} (4) is verified for all $i\ge 1$.

The proposition is thus proved.
\end{proof}

\emph{From now until the end of this section, we assume $k\ge 2$ is  even and put $m=k/2$.}

\medskip

\newpara\label{UL5.3} Suppose that $L$ satisfies the equivalent conditions in Proposition\;\ref{UL5.2}. In particular, we have $L_1= L_{\le -1}\subseteq L_{(0)}$ and $\dim L_1\ge k$. By \cite[(93:11), (93:15) and (93:18) (ii)]{OMeara00}, we also have
\begin{equation}\label{UL5.3.1}
  \begin{split}
 \text{either}\quad   L_1&\cong 2^{-1}A(0,\,0)\bot\cdots\bot 2^{-1}A(0,\,0)\bot 2^{-1}A(0,\,0)\,,\\
     \text{or}\quad L_1&\cong 2^{-1}A(0,\,0)\bot\cdots\bot 2^{-1}A(0,\,0)\bot 2^{-1}A(2,\,2\rho)\,.
  \end{split}
\end{equation}In particular, $FL_1$ contains a hyperbolic space of dimension $\ge \dim L_1-2$.

Let $\ell=I_{-1}\bot I_1\bot P_0\bot P_1$ be an arbitrary dominant lattice of dimension $k$. We use Theorem\;\ref{UL3.4} to examine the representability of $\ell$ by $L$.

For $i<-2$, we have $\ell_{\le i}=\ell_{[i]}=0$ and $\Delta_i(\ell)=\Delta_i(L)=0$. In this case all the conditions in Theorem\;\ref{UL3.4} hold trivially.

Now consider the case $i=-2$. We have $\ell_{\le -2}=0$, $\ell_{[-2]}=I_{-1}$,
\[
\Delta_{-2}(\ell)=\begin{cases}
  \cO_F \quad & \text{ if } P_0\neq 0\,\\
  0 \quad & \text{ otherwise}\,
\end{cases}\quad\text{and}\quad \Delta_{-2}(L)=\begin{cases}
  \cO_F \quad & \text{ if } \fn(L_2)=\cO_F\,,\\
  0 \quad & \text{ otherwise}\,.
\end{cases}
\]Thus, it is easily seen that conditions (2) and (3) in Theorem\;\ref{UL3.4} hold (for $i=-2$). Condition \eqref{UL3.4} (1) can be translated into the following three conditions:

\begin{enumerate}
  \item[(a)] $FI_{-1}\rep  FL_{(0)}$;
  \item[(b)] if $P_0\neq 0$, then $\cO_F\rep  FL_{(0)}/^\bot FI_{-1}$;
  \item[(c)] if $\fn(L_2)=\cO_F$, then $\cO_F\rep  FL_{(0)}/^\bot FI_{-1}$.
\end{enumerate}
Condition \eqref{UL3.4} (4) can be rephrased as follows:
\begin{enumerate}
  \item[(d)] $FI_{-1}\rep FL_1\bot \dgf{1}$ and $(FL_1\bot \dgf{1})/^{\bot} FI_{-1}$ represents $1$ or $\Delta$.
\end{enumerate}

If $\dim L_1\ge k+2$, then all the above four conditions hold for dimensional reasons. Indeed, $\dim  FL_{(0)}-\dim FI_{-1}\ge 3$, so we can apply \cite[(63:21)]{OMeara00} to get (a), and (b) and (c) follow from Proposition\;\ref{UL2.2}. For (d) we also use Lemma\;\ref{UL2.4} (3).

Now we claim that when $\dim L_1=k$, in order that all the four conditions be satisfied by all (dominant $k$-dimensional) $\ell$ it is sufficient and necessary that $\fn(L_2)=\cO_F$. (Since $\frs(L_2)\subset \frs(L_1)=2^{-1}\cO_F$, this is equivalent to $\cO_F\subseteq \fn(L_2)$ and it implies $\frs(L_2)=\cO_F$.)

To prove the claim, note that by considering the $k$-dimensional lattices
\[
2^{-1}A(0,\,0)\bot \cdots\bot 2^{-1}A(0,\,0)\quad\text{ and }\quad2^{-1}A(0,\,0)\bot \cdots\bot 2^{-1}A(0,\,0)\bot2^{-1}A(2,\,2\rho)\,,
 \]we conclude from (a) that $ FL_{(0)}$ represents both $\bH^m$ and $\bH^{m-1}\bot \dgf{1,\,-\Delta}$. Therefore, $\dim L_{(0)}>k=\dim L_1$, that is, $\cO_F\subseteq \fn(L_2)$.

Conversely, let us suppose $\frs(L_2)=\fn(L_2)=\cO_F$ and prove that (a)--(d) are satisfied. By \eqref{UL5.3.1}, $FL_1$ is either $\bH^m$ or $\bH^{m-1}\bot \dgf{1,\,-\Delta}$.
Similarly, $FI_{-1}$ is always contained in and may be equal to one of the two spaces $\bH^m,\,\bH^{m-1}\bot \dgf{1,\,-\Delta}$. For any $\veps\in \cO_F^*$ we have
$\dgf{1,\,-\Delta}\bot \dgf{\veps}\cong \bH\bot \dgf{\Delta\veps}$. So in any case $FL_1\bot\dgf{\veps}$ contains $FI_{-1}$ and $(FL_1\bot \dgf{\veps})/^{\bot}FI_{-1}$ represents
$\veps$ or $\Delta\veps$. From this we see that (a)--(d) are all satisfied. Our claim is thus proved.

By the above discussions, we may assume
\begin{equation}\label{UL5.3.2}
\begin{split}
  &\frs(L_1)=2^{-1}\cO_F=2^{-1}\fn(L_1)\,,\;\dim L_1\ge k\,,\\
   \text{ and if }&\dim L_1=k\,,\;\text{ then }\;\frs(L_2)=\fn(L_2)=\cO_F
\end{split}
\end{equation}
in the study of the $k$-universality property.

\

\newpara\label{UL5.4} Suppose that $L$ satisfies \eqref{UL5.3.2}, and let $\ell=I_{-1}\bot I_0\bot P_0\bot P_1$ be an arbitrary dominant $k$-dimensional lattice.
We have seen that for all $i\le -2$ the conditions in Theorem\;\ref{UL3.4} are all satisfied.

Notice that $\ell_{\le -1}=I_{-1}$, $\ell_{[-1]}=I_{-1}\bot I_0$ and
\[
\Delta_{-1}(\ell)=\begin{cases}
  \cO_F\quad & \text{ if } P_0\neq 0\,,\\
  2\cO_F\quad & \text{ if } P_0=0\neq P_1\,,\\
  0 \quad & \text{ other cases}\,.
\end{cases}
\]
So for $i=-1$, conditions \eqref{UL3.4} (1) and (2)  assert the following:
\begin{enumerate}
  \item[(a)] $FI_{-1}\bot FI_0\rep  FL_{(1)}$, $\Delta_{-1}(\ell)\rep  FL_{(1)}/^\bot (FI_{-1}\bot FI_0)$ and $\Delta_{-1}(L)\rep  FL_{(1)}/^\bot (FI_{-1}\bot FI_0)$;
  \item[(b)] if $ FL_{(1)}/^\bot (FI_{-1}\bot FI_0)\cong\bH$, $P_0=0$  and $P_1\neq 0$, then $L$ has no proper unimodular component.
\end{enumerate}
Condition \eqref{UL3.4} (3) now requires that $FI_{-1}\rep  FL_{(0)}\bot \dgf{2}$ and that its orthogonal complement represents $2$ or $2\Delta$. This holds automatically because we have seen that $FI_{-1}\rep  FL_{(0)}$ in \eqref{UL5.3}.

Condition \eqref{UL3.4} (4) means that $FI_{-1}\bot FI_0\rep  FL_{\le 0}\bot \dgf{2}$ and that its orthogonal complement represents represents $2$ or $2\Delta$. This is also guaranteed by \eqref{UL5.3.2}. Indeed, if $\dim  FL_{\le 0}\bot \dgf{2}-\dim FI_{-1}\bot FI_0<3$, then $\dim I_{-1}\bot I_0=\dim \ell=\dim L_1=k$ and $ FL_{\le 0}\bot \dgf{2}=FL_1\bot \dgf{\veps,\,2}$ for some $\veps\in \cO_F^*$ according to \eqref{UL5.3.2}. Moreover, the determinants of the spaces  $FI_{-1}\bot FI_0$ and $FL_1$ are units. So $FL_1\bot \dgf{\veps,\,2}$ represents $FI_{-1}\bot FI_0$ by \cite[(63:21)]{OMeara00} and the space $\big(FL_1\bot \dgf{\veps,\,2}\big)/^\bot (FI_{-1}\bot FI_0)$ represents $2$ or $2\Delta$ by \cite[(63:11)]{OMeara00}.

We have $\ell_{\le 0}=\ell_{[0]}=I_{-1}\bot I_0\bot P_0$ and $\Delta_0(\ell)=2\cO_F$ if $P_1\neq 0$, or $\Delta_0(\ell)=0$ if $P_1=0$. So in the case $i=0$ condition \eqref{UL3.4} (2) is trivially verified and \eqref{UL3.4} (1) reads
\begin{enumerate}
  \item[(c)] $FI_{-1}\bot FI_0\bot FP_0$ is represented by $ FL_{(2)}$, and the orthogonal complement represents the ideals $\Delta_{0}(\ell)$ and $\Delta_{0}(L)$.
\end{enumerate}
Condition \eqref{UL3.4} (4) is implied by \eqref{UL3.4} (3), and the latter can be translated into the following condition:
\begin{enumerate}
  \item[(d)] $FI_{-1}\bot FI_0\bot FP_0$ is represented by $  FL_{(1)}\bot \dgf{1}$ and the orthogonal complement represents $1$ or $\Delta$.
\end{enumerate}

For $i\ge 1$ we have $\ell_{\le i}=\ell_{[i]}=\ell$ and $\Delta_i(\ell)=0$. Condition \eqref{UL3.4} (2) holds trivially in this case and \eqref{UL3.4} (3) implies \eqref{UL3.4} (4). The only remaining conditions to check can be rephrased as follows:
\begin{enumerate}
  \item[(e)] $F\ell\rep  FL_{(3)}$;
  \item[(f)] for every $i\ge 1$, $\Delta_i(L)\rep  FL_{(i+2)}/^\bot F\ell$;
  \item[(g)] $F\ell$ is represented by $ FL_{(2)}\bot \dgf{2}$ and the orthogonal complement represents $2$ or $2\Delta$.
\end{enumerate}

\begin{prop}\label{UL5.5}
Suppose that $\dim L_1\ge k+2$, $\frs(L_1)=2^{-1}\cO_F$ and $\fn(L_1)=\cO_F$.

Then $L$ is $k$-universal if and only if one of the following holds:

    \begin{enumerate}
      \item $\dim L_1\ge k+4$;
      \item $\dim L_1=k+2$, $d_{\pm}(FL_1)\in F^{*2}$, and if $k>2$, then $2\cO_F\subseteq \fn(L_2)$;
      \item $\dim L_1=k+2$, $d_{\pm}(FL_1)\in \Delta F^{*2}$, and $2\cO_F\subseteq \fn(L_2)$.
    \end{enumerate}
\end{prop}
\begin{proof}
If $\dim L_1\ge k+4$ or if $\dim L_1=k+2$ and $2\cO_F\subseteq \fn(L_2)$, then we have $\dim L_{(1)}\ge k+3$ and hence all the conditions (a)--(g) in \eqref{UL5.4} can be checked easily by a dimension argument (using \cite[(63:21)]{OMeara00}, Proposition\;\ref{UL2.2}, Lemma\;\ref{UL2.4} (3), etc.).

In the rest of the proof let us assume $\dim L_1=k+2$ and $\fn(L_2)\subset 2\cO_F$ (thus $L_{(1)}=L_1$). Under this assumption we need to show that $L$ is not $k$-universal if $k>2$, and that for $k=2$, $L$ is $k$-universal if and only if $d_{\pm}(FL_1)= d(FL_1)\in F^{*2}$.

By \eqref{UL5.3.1} we have either $FL_1\cong \bH^{m+1}$ (in which case $d_{\pm}(FL_1)\in F^{*2}$), or $FL_1\cong \bH^m\bot \dgf{1,\,-\Delta}$ (in which case $d_{\pm}(FL_1)\in \Delta F^{*2}$). Similarly,  the space $FI_{-1}\bot FI_0$ is always contained in and may be equal to one of the following spaces:
\[
\begin{split}
  U_1:=\bH^m\,, \;U_2&:=\bH^{m-1}\bot \dgf{1,\,-\Delta}\,,\;U_3:= \bH^{m-1}\bot \dgf{2,\,-2\Delta}\;,\\
\text{ and if } m>1\,,\; U_4&:=\bH^{m-2}\bot \dgf{1,\,-\Delta}\bot\dgf{2,\,-2\Delta}\,.
\end{split}
\]The space $\bH^m\bot \dgf{1,\,-\Delta}$ does not represent $U_3$; the space $\bH^{m+1}$  represents all the three spaces $U_1,\,U_2,\,U_3$, but does not represent $U_4$ if $m>1$ (i.e. $k>2$). Therefore, by condition \eqref{UL5.4} (a), if $L$ is $k$-universal, we must have $FL_1\cong \bH^{m+1}$ and $k=2$.

Conversely, suppose  $k=2$ and $FL_1\cong \bH^{m+1}=\bH^2$. Note that $\Delta_{-1}(L)=0$ and that if $\Delta_{-1}(\ell)\neq 0$, then $\dim FI_{-1}\bot FI_0\le k-1=\dim FL_1-3$. So conditions \eqref{UL5.4} (a) and (b) hold. By a dimension argument we can check \eqref{UL5.4} (d), (f) and (g). Clearly, the space $FL_1=\bH^2$ represents all binary spaces. In particular, \eqref{UL5.4} (e) holds and the space $U:=FI_{-1}\bot FI_0\bot FP_0$ is represented by $FL_1$. If $\Delta_0(\ell)\neq 0$, then $P_1\neq 0$ and $\dim U\le 1=\dim FL_1-3$; if $\Delta_0(L)\neq 0$, then $\dim L_{(2)}\ge \dim L_1+1\ge \dim U+3$. Hence \eqref{UL5.4} (c) also holds. The proposition is thus proved.
\end{proof}

\begin{lemma}\label{UL5.6}
Suppose that $L_1$ is  improper $2^{-1}$-modular of dimension $k$ and that $L_2$ is binary proper unimodular.

  \begin{enumerate}
    \item Conditions (b), (d) and (g) of $\eqref{UL5.4}$ hold for all dominant $k$-dimensional lattices $\ell$.
    \item Condition $\eqref{UL5.4}$ (a) holds for all dominant $k$-dimensional lattices $\ell$ if and only if one of the following holds:

    \begin{enumerate}
      \item[(i)] $\frs(L_3)=\fn(L_3)=2\cO_F$;
      \item[(ii)] $d_{\pm}(FL_2)=-d(FL_2)\notin F^{*2}\cup \Delta F^{*2}$.
    \end{enumerate}
  \end{enumerate}
\end{lemma}
\begin{proof}
First note that $L_{(0)}=L_1\bot L_2$ and hence $\dim L_{(0)}=k+2$. So \eqref{UL5.4} (b), (d) and (g) hold for dimensional reasons. If $\fn(L_3)=2\cO_F$, then $\dim L_{(1)}\ge k+3$ and \eqref{UL5.4} (a) also holds by dimensional reasons.

Now let us assume $\fn(L_3)\subset 2\cO_F$, i.e., $L_{(1)}=L_1\bot L_2$. Since $\dgf{1,\,-\Delta}$ represents all units, from \eqref{UL5.3.1} we see that $ FL_{(1)}=FL_1\bot FL_2=\bH^m\bot \dgf{\eta_1,\,\eta_2}$ for some $\eta_i\in\cO_F^*$. Let the spaces $U_1,\cdots, U_4$ be as in the proof of Proposition\;\ref{UL5.5}. The space $\bH^m\bot \dgf{\eta_1,\,\eta_2}$ represents both $U_1$ and $U_2$ (noticing that $\dgf{1,\,-\Delta}\cong \dgf{\eta_1,\,-\Delta\eta_1}$). By \cite[(63:21)]{OMeara00}, it represents $U_3$ if and only if $- d(FL_2)=-\eta_1\eta_2\notin \Delta F^{*2}$, and when $k>2$, it represents $U_4$ if and only if $- d(FL_2)=-\eta_1\eta_2\notin  F^{*2}$.

Let us show that we must have $- d(FL_2)\notin F^{*2}$ also in the case $k=2$. To see this we may choose $\ell$ such that $FI_{-1}\bot FI_0=U_3$. If $- d(FL_2)\in F^{*2}$, then the space $V:= FL_{(1)}/^\bot (FI_{-1}\bot FI_0)$ would be $\dgf{2,\,-2\Delta}$. But then the condition $\cO_F=\Delta_{-1}(L)\rep V$ in \eqref{UL5.4} (a) fails. So we have shown  that (ii) is a necessary condition for \eqref{UL5.4} (a).

Conversely, assuming (ii) we have seen that $FI_{-1}\bot FI_0$ is represented by $ FL_{(1)}$, and that the space $V:= FL_{(1)}/^\bot (FI_{-1}\bot FI_0)$ is  either of dimension $\ge 3$ or binary with $- d(V)\in \cO_F^*\setminus \Delta\cO_F^{*2}$. We have $\Delta_{-1}(\ell)\rep V$ for dimensional reasons, and $\cO_F=\Delta_{-1}(L)\rep V$ by Proposition\;\ref{UL2.2} or Corollary\;\ref{UL2.6}.
\end{proof}

\begin{prop}\label{UL5.7}
 Suppose that $L_1$ is  improper $2^{-1}$-modular of dimension $k$ and that $L_2$ is proper unimodular of dimension $\ge 2$.

Then $L$ is $k$-universal if and only if one of the following holds:

    \begin{enumerate}
      \item $\dim L_2\ge 3$;
       \item $\dim L_2=2$, and $\frs(L_3)=\fn(L_3)=2\cO_F$;
      \item $\dim L_2=2$, $d_{\pm}(FL_2)\notin F^{*2}\cup \Delta F^{*2}$ and $\fn(L_3)=4\cO_F$.
    \end{enumerate}
\end{prop}
\begin{proof}
  If $\dim L_2\ge 3$, or if $\dim L_2=2$ and $\fn(L_3)=2\cO_F$, then $\dim L_{(1)}\ge k+3$ and all the conditions (a)--(g) in \eqref{UL5.4} are verified by a dimension argument. Hence, we may assume $\dim L_2=2$ and $\fn(L_3)\subset 2\cO_F$, so that $L_{(1)}=L_1\bot L_2$. By Lemma\;\ref{UL5.6}, we may further assume $d_{\pm}(FL_2)=- d(FL_2)\notin F^{*2}\cup \Delta F^{*2}$. Under these assumptions we need to prove that $L$ satisfies conditions \eqref{UL5.4} (c), (e) and (f) if and only if $\fn(L_3)\supseteq 4\cO_F$.

  The ``if'' part follows easily from the usual dimension argument. To prove the ``only if'' part, first note that  $FL_1\bot FL_2=\bH^m\bot \dgf{\eta_1,\,\eta_2}$ with $\eta_i\in\cO_F^*$. Since $-\eta_1\eta_2=- d(FL_2)\notin F^{*2}\cup \Delta F^{*2}$, by \cite[p.202, Lemma\;3]{Hsia75Pacific} we can find $\veps_1\in \cO_F^*$ such that $\veps_1\notin Q(\dgf{\eta_1,\,\eta_2})=Q(\eta_1\dgf{1,\,\eta_1\eta_2})$. Then $\dgf{-\veps_1,\,\eta_1,\,\eta_2}$ is anisotropic and by Lemma\;\ref{UL2.4} (1), there exists $\veps_2\in \cO_F^*$ such that $\veps_2\notin Q(\dgf{-\veps,\,\eta_1,\,\eta_2})$. Thus $\dgf{\veps_1,\,\veps_2}$ is not represented by $\bH\bot\dgf{\eta_1,\,\eta_2}=\dgf{\veps_1,\,-\veps_1}\bot\dgf{\eta_1,\,\eta_2}$. Then the space $U_5:=\bH^{m-1}\bot \dgf{\veps_1,\,\veps_2}$ is not represented by $FL_1\bot FL_2$. Taking  $\ell$ such that $FI_{-1}\bot FI_0\bot FP_0=U_5$, we deduce from \eqref{UL5.4} (c) that $L_1\bot L_2\subset L_{(2)}$, that is, $\fn(L_3)\supseteq 4\cO_F$.

  The proposition is thus proved.
\end{proof}

\begin{lemma}\label{UL5.8}
 Suppose that $L_1$ is  improper $2^{-1}$-modular of dimension $k$ and that $L_2$ is unimodular of dimension $1$.

Then the following statements are equivalent:

\begin{enumerate}
  \item Condition $\eqref{UL5.4}$ (a) holds for all dominant $k$-dimensional lattices $\ell$.
  \item $\frs(L_3)=\fn(L_3)=2\cO_F$.
  \item Conditions $\eqref{UL5.4}$ (a), (b), (d) and (g) hold for all dominant $k$-dimensional lattices $\ell$.
\end{enumerate}
\end{lemma}
\begin{proof}
  We have $L_2=\dgf{\eta}$ for some $\eta\in \cO_F^*$ and hence $ FL_{(0)}=FL_1\bot FL_2=\bH^m\bot \dgf{\eta'}$, where $\eta'\in \{\eta\,,\,\Delta\eta\}$. The space $U_3=\bH^{m-1}\bot\dgf{2,\,-2\Delta}$ is not represented by $ FL_{(0)}$. Choosing $\ell$ such that $FI_{-1}\bot FI_0=U_3$, we see that \eqref{UL5.4} (a) forces $L_{(0)}\subset L_{(1)}$, i.e., $\frs(L_3)=\fn(L_3)=2\cO_F$.

  Conversely, if $\frs(L_3)=\fn(L_3)=2\cO_F$, then conditions (b), (d) and (g) of \eqref{UL5.4} hold for dimensional reasons.
\end{proof}

\begin{prop}\label{UL5.9}
 Suppose that $L_1$ is  improper $2^{-1}$-modular of dimension $k$ and that $L_2$ is unimodular of dimension $1$.

Then $L$ is $k$-universal if and only if $\frs(L_3)=\fn(L_3)=2\cO_F$ and one of the following holds:

  \begin{enumerate}
   \item $\dim L_3\ge 2$.
    \item $\dim L_3=1$ and $8\cO_F\subseteq\fn(L_4)$.
  \end{enumerate}
\end{prop}
\begin{proof}
By Lemma\;\ref{UL5.8} we may assume $\frs(L_3)=\fn(L_3)=2\cO_F$. We need to show that conditions (c), (e) and (f) of \eqref{UL5.4} hold for all dominant $k$-dimensional lattices $\ell$ if and only if one of the two conditions in the proposition is satisfied.

If $\dim L_3\ge 2$, then we have $\dim L_{(1)}\ge k+3$ and hence \eqref{UL5.4} (c), (e) and (f) follow by a dimension argument. So let us assume $\dim L_3=1$.

We can write $L_2=\dgf{\eta}$ and $L_3=\dgf{2\veps}$ for some $\eta,\,\veps\in\cO_F^*$. Then $ FL_{(1)}=FL_1\bot FL_2\bot FL_3=\bH^m\bot \dgf{\eta',\,2\veps}$, where $\eta'\in \{\eta,\,\Delta\eta\}$. The space $U:=FI_{-1}\bot FI_0\bot FP_0$ has dimension $\le k$ and $ d(U)\in \cO_F^*$. So $U$ is represented by $ FL_{(1)}$ by \cite[(63:21)]{OMeara00} and the orthogonal complement $V:= FL_{(1)}/^{\bot}U$ is at least two dimensional and its determinant is a uniformizer (up to squares). By \cite[(63:11)]{OMeara00} (or Proposition\;\ref{UL2.2}), $V$ represents $\Delta_0(\ell)$ and $\Delta_0(L)$. Thus, \eqref{UL5.4} (c) is verified.

It remains to show that \eqref{UL5.4} (e) and (f) hold for all $\ell$ if and only if $8\cO_F\subseteq \fn(L_4)$.

Indeed, the space $ FL_{(1)}=FL_1\bot FL_2\bot FL_3=\bH^m\bot\dgf{\eta',\,2\veps}$ does not represent $\bH^{m-1}\bot\dgf{\Delta\eta',\,2\Delta\veps}$, which may be $F\ell$. Therefore, \eqref{UL5.4} (e) implies $L_{(1)}\subseteq L_{(3)}$, i.e. $8\cO_F\subseteq \fn(L_4)$. Conversely, when $8\cO_F\subseteq \fn(L_4)$ we have $\dim  FL_{(3)}\ge k+3$, so that \eqref{UL5.4} (e) and (f) are satisfied for dimensional reasons. The proposition is thus proved.
\end{proof}

Combining Propositions\;\ref{UL5.5}, \ref{UL5.7} and \ref{UL5.9} proves Theorem\;\ref{UL1.3}.

\begin{prop}\label{UL5.10}
Suppose $L$ is classic. Then the following conditions are equivalent:

  \begin{enumerate}
    \item Every classic basic lattice of dimension $k$ has a lower type than $L$.
    \item $L_1$ is proper unimodular of dimension $\ge k+1$.
  \end{enumerate}
\end{prop}
\begin{proof}
  The key ideas are similar to those in the proof of Proposition\;\ref{UL5.2}. We only prove (1)$\Rightarrow$(2) here.

  Let $\ell=I_0\bot P_0\bot I_1\bot P_0$ be an arbitrary classic basic lattice of dimension $k$. Then
  \[
  \ell_{\le i}=\begin{cases}
    0 \quad & \text{ if } i<0\,,\\
    I_0\bot P_0 \quad & \text{ if } i=0\,,\\
    \ell \quad & \text{ if } i\ge 1\,.
  \end{cases}
  \]
  If condition (1) of \eqref{UL3.2} holds for all $\ell$, we must have $\dim  L_{\le 0}\ge k$. Since $L$ is assumed to be classic, this means that $L_1$ is unimodular of dimension $\ge k$. Assuming this, we see that the case $i=-1$ of condition \eqref{UL3.2} (3) requires $L$ to have a proper unimodular component. So we may assume that $L_1$ is proper unimodular of dimension $\ge k$.

  Next consider the case where $\ell$ itself is improper unimodular. (This is possible since $k$ is even.) Then the case $i=0$ of \eqref{UL3.2} (3) forces $\dim L_1\ge k+1$. This finishes the proof.
\end{proof}

\newpara\label{UL5.11} In the rest of this section, we assume that $L$ satisfies the equivalent conditions in Proposition\;\ref{UL5.10}. That is, $L_1$ is proper unimodular of dimension $\ge k+1$. So $L_1=L_{(0)}= L_{\le 0}$. Let $\ell=I_0\bot P_0\bot I_1\bot P_1$ be an arbitrary classic basic lattice of dimension $k$. We use Theorem\;\ref{UL3.4} to determine when $\ell$ is represented by $L$.

For $i<-2$, we have $\ell_{[i]}=\ell_{\le i}=0$ and $\Delta_i(\ell)=\Delta_i(L)=0$. For $i=-2$, $\ell_{[i]}=\ell_{\le i}=0$, $\Delta_i(L)=\cO_F$ and $\Delta_i(\ell)$ is 0 or $\cO_F$. In these cases all the conditions in Theorem\;\ref{UL3.4} are trivially verified.

For $i=-1$, we have $\ell_{[-1]}=I_0$, $\ell_{\le -1}=0$ and $\Delta_{-1}(L)=\cO_F$. In particular, \eqref{UL3.4} (3) is trivially checked. If $\Delta_{-1}(\ell)\neq 0$, then $\dim I_0<k$ and for a parity reason, $\dim I_0\le k-2\le \dim L_1-3$. This shows that \eqref{UL3.4} (2) holds automatically in this case, and that \eqref{UL3.4} (1) is equivalent to the following two conditions:

\begin{enumerate}
  \item[(a)] $FI_0\rep  FL_{(1)}$;
  \item[(b)] $\cO_F\rep  FL_{(1)}/^\bot FI_0$.
\end{enumerate}

Condition \eqref{UL3.4} (4) requires that $FI_0\rep FL_1\bot\dgf{2}$ and the orthogonal complement represents $2$ or $2\Delta$. But this holds automatically. Indeed, if $\dim FL_1\bot\dgf{2}-\dim FI_0<3$, then $\dim FL_1\bot\dgf{2}-\dim FI_0=2$, and by \cite[(63:11) and (63:21)]{OMeara00}, we get the result from a determinant argument.

Next consider the case $i=0$. We have $\ell_{[0]}=I_0\bot P_0\bot I_1$, $\ell_{\le 0}=I_0\bot P_0$, $\Delta_0(\ell)=2\cO_F$ if $P_1\neq 0$ or $\Delta_0(\ell)=0$ otherwise; and $\Delta_0(L)\subseteq 2\cO_F$. Thus \eqref{UL3.4} (2) holds automatically. The other three conditions in \eqref{UL3.4} now assert the following

\begin{enumerate}
  \item[(c)] $FI_0\bot FP_0\bot FI_1\rep  FL_{(2)}$ and the orthogonal complement represents $\Delta_0(\ell)$ and $\Delta_0(L)$;
  \item[(d)] $FI_0\bot FP_0\rep  FL_{(1)}\bot\dgf{1}$ and the orthogonal complement represents $1$ or $\Delta$;
  \item[(e)] $FI_0\bot FP_0\bot FI_1\rep  FL_{\le 1}\bot\dgf{1}$ and the orthogonal complement represents $1$ or $\Delta$.
\end{enumerate}

Finally, suppose $i\ge 1$. Then $\ell_{[i]}=\ell_{\le i}=\ell$ and $\Delta_i(\ell)=0$. So there is no need to check \eqref{UL3.4} (2). Condition \eqref{UL3.4} (1) can be translated as the following two conditions:

\begin{enumerate}
  \item[(f)] $F\ell\rep  FL_{(3)}$;
  \item[(g)] for all $i\ge 1$, $ FL_{(i+2)}/^\bot F\ell$ represents $\Delta_i(L)$.
\end{enumerate}

Condition \eqref{UL3.4} (3) implies \eqref{UL3.4} (4). For even $i\ge 1$, \eqref{UL3.4} (3) is a consequence of our condition (f) above. For odd $i\ge 1$, \eqref{UL3.4} (3) is equivalent to the following condition:

\begin{enumerate}
  \item[(h)] $FI_0\rep  FL_{(2)}\bot \dgf{2}$ and the orthogonal complement represents $2$ or $2\Delta$.
\end{enumerate}

\begin{prop}\label{UL5.12}
  Suppose that $L_1$ is  proper unimodular of dimension $\ge k+2$.

Then $L$ is classic $k$-universal if and only if one of the following holds:

    \begin{enumerate}
      \item $\dim L_1\ge k+3$;
       \item $\dim L_1=k+2$, and $\frs(L_2)=\fn(L_2)=2\cO_F$;
      \item $\dim L_1=k+2$, $d_{\pm}(FL_1)\notin F^{*2}\cup \Delta F^{*2}$ and $\fn(L_2)=4\cO_F$.
    \end{enumerate}
\end{prop}
\begin{proof}
  If $\dim L_1\ge k+3$ or if $\dim L_1=k+2$ and $\fn(L_2)=2\cO_F$, then all the conditions (a)--(h) in \eqref{UL5.11} hold for dimensional reasons. So we shall assume $\dim L_1=k+2$ and $\fn(L_2)\subseteq 4\cO_F$. (Thus $L_1=L_{(0)}=L_{(1)}$.) Under these assumptions we need to prove that $L$ is classic $k$-universal if and only if $d_{\pm}(FL_1)\notin F^{*2}\cup \Delta F^{*2}$ and $4\cO_F\subseteq\fn(L_2)$. Since conditions (d), (e) and (h) of \eqref{UL5.11} are easily verified by a dimension argument, it remains to consider conditions (a)--(c), (f) and (g) in \eqref{UL5.11}.

  As $\ell$ varies over all classic basic lattices of dimension $k$, $FI_0$ is always contained in and may be equal to $\bH^m$ or $\bH^{m-1}\bot \dgf{2,\,-2\Delta}$. So \eqref{UL5.11} (a) means precisely that $FL_1= FL_{\le 1}$ should contain both $\bH^m$ and $\bH^{m-1}\bot \dgf{2,\,-2\Delta}$. This is the case if $d_{\pm}(FL_1)\notin F^{*2}\cup \Delta F^{*2}$, by \cite[(63:21)]{OMeara00}.

  Suppose that $d_{\pm}(FL_1)\in F^{*2}$. Then the condition $\bH^m\rep FL_1$ implies that $FL_1\cong \bH^{m+1}$. But if we choose $\ell$ such that $FI_0=\bH^{m-1}\bot \dgf{2,\,-2\Delta}$, then \eqref{UL5.11} (b) fails because in this case $FL_1/^\bot FI_0\cong \dgf{2,\,-2\Delta}$.

  Similarly, if $d_{\pm}(FL_1)\in \Delta F^{*2}$, then the inclusion $\bH^{m-1}\bot\dgf{2,\,-2\Delta}\subseteq FL_1$ forces $FL_1\cong \bH^m\bot \dgf{2,\,-2\Delta}$, and thus \eqref{UL5.11} (b) fails when $FI_0=\bH^m$.

  We may thus conclude that if \eqref{UL5.11} (a) and (b) both hold for all $\ell$, then we must have $d_{\pm}(FL_1)\notin F^{*2}\cup\Delta F^{*2}$. So we further assume this holds in the rest of the proof.

  Now the space $V:= FL_{(1)}/^\bot FI_0$ is either of dimension 3 or is binary with $-d(V)\in \cO_F^*\setminus (\cO_F^{*2}\cup \Delta \cO_F^{*2})$. So \eqref{UL5.11} (b) is satisfied by Proposition\;\ref{UL2.2} or Corollary\;\ref{UL2.6}.

  It remains to show (under all the previous assumptions) that conditions \eqref{UL5.11} (c), (f) and (g) hold for all $\ell$ if and only if $4\cO_F\subseteq \fn(L_2)$, i.e., $L_1\subset L_{(2)}$.

  Indeed, if $L_1\subset L_{(2)}$, then $\dim L_{(2)}\ge k+3$ and the result follows by a dimension argument. Conversely, let us assume $L_1=L_{(2)}$. We need only to show that in this case \eqref{UL5.11} cannot hold for all $\ell$. To this end, first notice that, as it may happen that $FP_0=\bH^m$, a necessary condition for \eqref{UL5.11} (c) to hold for all $\ell$ is that $\bH^m\subseteq FL_1$. By a discriminant analysis, we find that $FL_1=\bH^m\bot \dgf{a,\,-ac}$, where $c=d_{\pm}(FL_1)\in \cO_F^*\setminus (\cO_F^{*2}\cup \Delta \cO_F^{*2})$ and $a\in \cO^*_F\cup 2\cO_F^*$. But it may also happen that $FP_0=\bH^{m-1}\bot\dgf{\veps_1,\,\veps_2}$ for all $\veps_1,\,\veps_2\in \cO_F^*$. From \eqref{UL5.11} (c) we also get $\dgf{\veps_1,\,\veps_2}\rep \bH\bot \dgf{a,\,-ac}$, which is equivalent to saying that
  $\dgf{a,\,-ac,\,\veps_1,\,\veps_2}$ is isotropic, for all $\veps_1,\,\veps_2\in \cO_F^*$. But we can always find $\veps\in \cO_F^*$ such that $\dgf{1,\,-c,\,-a\veps,\,ac\veps}$ is anisotropic. If $a\in 2\cO_F^*$, this follows from Lemma\;\ref{UL2.5};  if $a\in \cO_F^*$, we can apply \cite[p.202, Lemma\;3]{Hsia75Pacific}. Taking $\veps_1=\veps$ and $\veps_2=-c\veps$ yields a contradiction.  
\end{proof}

\begin{prop}\label{UL5.13}
  Suppose $L_1$ is proper unimodular of dimension $k+1$.

  Then $L$ is $k$-universal if and only if $\frs(L_2)=\fn(L_2)=2\cO_F$ and one of the following holds:

  \begin{enumerate}
    \item $\dim L_2\ge 2$;
    \item $\dim L_2=1$ and $8\cO_F\subseteq \fn(L_3)$.
  \end{enumerate}
\end{prop}
\begin{proof}
  As in the proof of Proposition\;\ref{UL5.12}, in order to have \eqref{UL5.11} (a) satisfied for all $k$-dimensional classic basic lattices $\ell$, it is sufficient and necessary that $ FL_{(1)}$ represent both $U_1=\bH^m$ and $U_3=\bH^{m-1}\bot \dgf{2,\,-2\Delta}$. Since $d(FL_1)\in\cO_F^*$, the space $FL_1$ cannot represent both $U_1$ and $U_3$. So we must have $L_1\subset L_{(1)}$, i.e., $\frs(L_2)=\fn(L_2)=\cO_F$. If moreover $\dim L_2\ge 2$, then $\dim L_{(1)}\ge k+3$ and we see immediately that all the conditions (a)--(h) in \eqref{UL5.11} hold.

  In the rest of the proof, we assume that $\dim L_2=1$ with $\fn(L_2)=2\cO_F$. Thus $L_{(1)}=L_1\bot L_2= L_{\le 1}$ has dimension $k+2$ and $d( FL_{(1)})$ is a uniformizer up to squares. Conditions \eqref{UL5.11} (d), (e) and (h) hold for dimensional reasons, and \eqref{UL5.11} (a) and (b) follow from \cite[(63:21)]{OMeara00} combined with \cite[(63:11)]{OMeara00} (or \eqref{UL2.2}). If $\dim L_{(2)}-\dim I_0\bot P_0\bot I_1<3$, then $L_{(2)}=L_1\bot L_2$ and $\ell=I_0\bot P_0\bot I_1$, hence $d( FL_{(2)})/d(FI_0\bot FP_0\bot FI_1)$ is a uniformizer up to squares. So \eqref{UL5.11} (c) also follows from \cite[(63:11), (63:21)]{OMeara00} and Proposition\;\ref{UL2.2}.

  It remains to prove that (under the previous assumptions) \eqref{UL5.11} (f) and (g) hold for all $\ell$ if and only if $8\cO_F\subseteq \fn(L_3)$, i.e., $L_1\bot L_2\subset L_{(3)}$. The sufficiency follows by dimensional arguments. To prove the necessity, let us assume $L_{(3)}=L_1\bot L_2$. Then we need to show that \eqref{UL5.11} (f) does not hold for all $\ell$.

  Indeed, since $\dim FL_1\bot FL_2=k+2$ and $d(FL_1\bot FL_2)$ is a uniformizer up to squares, we have $ FL_{(3)}=FL_1\bot FL_2=\bH^m\bot \dgf{\eta,\,2\veps}$ for some $\eta,\,\veps\in \cO_F^*$. The space $F\ell$ can take the form $\bH^{m-1}\bot\dgf{\veps_1,\,\veps_2}$ for all $\veps_1,\,\veps_2\in \cO_F^*$. If \eqref{UL5.11} (f) holds for all $\ell$, we have $\dgf{\veps_1,\,\veps_2}\rep \bH\bot \dgf{\eta,\,2\veps}$ for all $\veps_1,\,\veps_2\in \cO_F^*$. But $\bH\bot \dgf{\eta,\,2\veps}$ does not represent $\dgf{\Delta\eta\,,\,2\Delta\veps}$ since otherwise we would have $\dgf{\eta,\,2\Delta\veps}$ by \cite[(63:21)]{OMeara00}. But this contradicts \cite[(63:11)]{OMeara00}, which claims that $\Delta\eta$ is not represented by $\dgf{\eta,\,2\veps}$.

  We have thus finished the proof of the proposition.
\end{proof}

Combining Propositions\;\ref{UL5.12} and \ref{UL5.13} proves Theorem\;\ref{UL1.4}.

\section{Classification of $k$-universal lattices for $k$ odd}\label{sec6}

In this section, let $k=2m+1\ge 3$ be an odd integer and let $L=L_1\bot \cdots\bot L_t$ be a Jordan splitting of an integral lattice.

\begin{prop}\label{UL6.1}
  The following conditions are equivalent:

  \begin{enumerate}
    \item Every dominant lattice of dimension $k$ has a lower type than $L$.
    \item $L_1$ is improper $2^{-1}$-modular, and one of the following holds:
    \begin{enumerate}
      \item $\dim L_1\ge k+1$;
      \item $\dim L_1=k-1$, $\dim L_2\ge 2$ and $\frs(L_2)=\fn(L_2)=\cO_F$;
      \item $\dim L_1=k-1$, $\dim L_2=1$, $\frs(L_2)=\fn(L_2)=\cO_F$ and $\frs(L_3)=\fn(L_3)=2\cO_F$.
    \end{enumerate}
  \end{enumerate}
\end{prop}
\begin{proof}
  (1)$\Rightarrow$(2). Let $\ell=I_{-1}\bot I_0\bot P_0\bot P_1$ be an arbitrary dominant lattice of dimension $k$. We assume that all such lattices $\ell$ have a lower type than $L$. Taking $\ell$ with $\dim I_{-1}=k-1$ and $\dim P_0=1$, we see from \eqref{UL3.2} (1) that $\dim  L_{\le -1}\ge k-1$ and $\dim  L_{\le 0}\ge k$. As we have assumed $L$ integral, this means that $L_1$ is improper $2^{-1}$-modular of (even) dimension $\ge k-1$ and that if $\dim L_1=k-1$, then $\frs(L_2)=\fn(L_2)=\cO_F$.

  Let us suppose $\dim L_1=k-1$ and $L_2$ is unimodular of dimension 1. Then we need to show that $\frs(L_3)=\fn(L_3)=2\cO_F$. In fact, we can choose $\ell$ such that $\dim I_{-1}=k-1$ and $\dim P_1=1$. Then $\dim \ell_{\le 0}=\dim  L_{\le 0}-1=k-1>0$. Therefore, by the case $i=0$ of \eqref{UL3.2} (4), $L$ should have a proper 2-modular component, i.e., $\frs(L_3)=\fn(L_3)=2\cO_F$ as desired.

  (2)$\Rightarrow$(1). First assume $\dim L_1=\dim  L_{\le -1}\ge k+1$. If $\dim\ell_{\le i}=\dim L_{\le i}$, then $i<-1$. Thus conditions (1)--(3) of \eqref{UL3.2} are trivially verified. If $\dim\ell_{\le i}=\dim L_{\le i}-1>0$, we must have $i\ge -1$, $\ell=\ell_{\le i}$ and $L_{\le i}= L_{\le -1}=L_1$ and hence $\ord(\fd_i(L))\equiv 0\pmod{2}$, $\ord(\fd_i(\ell))\equiv \dim P_1\pmod{2}$. If $i=-1$, then $\ord(\fd_i(\ell)\fd_i(L))\equiv 0=i+1\pmod{2}$ and $L$ has no proper $2^i$-modular component. If $i=0$, then $P_1=0$,
  $\ord(\fd_i(\ell)\fd_i(L))\equiv 0=i\pmod{2}$, and $\ell$ has no proper 2-modular component. If $i\ge 1$, then $\ell$ has no proper $2^{i+1}$-modular component and the condition $L_{\le i}=L_1$ implies that $L$ has no $2^i$-modular component. So we see that \eqref{UL3.4} (4) is also satisfied for all $\ell$.

  Next let us suppose that $L_1$ is improper $2^{-1}$-modular of dimension $k-1$ and that $L_2$ is proper unimodular of dimension $\ge 2$. Then $ L_{\le -1}=L_1$ and $ L_{\le 0}=L_1\bot L_2$. Clearly, condition \eqref{UL3.2} (1) is true for all $\ell$. If $\dim \ell_{\le i}=\dim L_{\le i}$, then either $i<-1$ and $\dim \ell_{\le i}=\dim L_{\le i}=0$, or $i=-1$ and $\dim \ell_{\le -1}=\dim I_{-1}=k-1$. In the former case, there is no need to check \eqref{UL3.2} (2), and \eqref{UL3.2} (3) is easily verified. If $i=-1$, then we find that $P_0=P_1=0$, and $\ord(\fd_i(\ell))\equiv \ord(\fd_i(L))\equiv 0\pmod{2}$. So \eqref{UL3.2} (2) and (3) hold again. If $\dim \ell_{\le i}=\dim L_{\le i}-1>0$, then $i\ge 0$ (noticing that $\dim \ell_{\le -1}=\dim I_{-1}$ and $\dim  L_{\le -1}=k-1$ are both even), $\dim \ell_{\le i}=k$ and $\dim L_{\le i}=k+1$. If $i=0$, we have
  $\ord(\fd_i(\ell)\fd_i(L))\equiv i=0\pmod{2}$ and the condition $\dim\ell_{\le i}=k$ shows that $\ell$ has no $2^{i+1}$-modular component. If $i>0$, then $\ell$ has no $2^{i+1}$-modular component and the condition $\dim L_{\le i}=k+1=\dim  L_{\le 0}$ implies that $L$ has no $2^i$-modular component. Thus we have verified \eqref{UL3.2} (4) for all $\ell$.

  Finally, let us consider the case where $\dim L_1=k-1$, $L_2$ is unimodular of dimension 1 and $L_3$ is proper 2-modular. In this case $\dim  L_{\le -1}=\dim L_1=k-1$, $\dim  L_{\le 0}=\dim L_1\bot L_2=k$ and $\dim  L_{\le 1}=\dim L_1\bot L_2\bot L_3\ge k+1$. Again \eqref{UL3.2} (1) is easily seen to hold. For $i<-1$, there is no need to check \eqref{UL3.2} (2) and (4), and \eqref{UL3.2} (3) holds trivially. Note that $\ord(\fd_{-1}(\ell)\fd_{-1}(L))$ is always even, $L$ has no proper $2^{-1}$-modular component but has a proper unimodular component. Therefore, \eqref{UL3.2} (2)--(4) are satisfied for $i=-1$. Similarly, conditions \eqref{UL3.2} (2) and (4) hold for $i=0$. If $\dim \ell_{\le 0}=\dim  L_{\le 0}=k$, then $P_1=0$ and $P_0\neq 0$ (since $k$ is odd but $\dim I_{-1}\bot I_0$ is even). So \eqref{UL3.2} (3) is also satisfied for $i=0$.

  If $i\ge 1$, then $\dim\ell_{\le i}=k<k+1\le \dim L_{\le i}$, so there is no need to check \eqref{UL3.2} (2) and (3). Suppose that $\dim\ell_{\le i}=k=\dim L_{\le i}-1$ for some $i\ge 1$. Then $L_{\le i}= L_{\le 1}$ has dimension $k+1$, $\ell$ has no $2^{i+1}$-modular component, and $L$ has no $2^i$-modular component unless $i=1$. Thus, to check \eqref{UL3.2} (4) is to check that if $\ord(\fd_1(\ell)\fd_1(L))\equiv 0\pmod{2}$, then $\ell$ has a proper $2$-modular component (i.e. $P_1\neq 0$). This is true because $\ord(\fd_1(L))\equiv \dim L_3=1\pmod{2}$ and $\ord(\fd_1(\ell))\equiv \dim P_1\pmod{2}$. This completes the proof.
\end{proof}

\newpara\label{UL6.3} Suppose that $L$ satisfies the equivalent conditions in Proposition\;\ref{UL6.1}. Let $\ell=I_{-1}\bot I_0\bot P_0\bot P_1$ be an arbitrary dominant lattice of dimension $k$.

For $i<-2$ we have $\ell_{\le i}=\ell_{[i]}=0$, $\Delta_i(\ell)=\Delta_i(L)=0$. So all the conditions in Theorem\;\ref{UL3.4} are trivially satisfied.

For $i=-2$, we have $\ell_{\le -2}=0$, $\ell_{[-2]}=I_{-1}$, and $\Delta_{-2}(\ell)$ is 0 or $\cO_F$ depending on whether $P_0=0$ or not. Conditions \eqref{UL3.4} (2) and (3) hold trivially, while \eqref{UL3.4} (1) asserts the following:
\begin{enumerate}
  \item[(a)] $FI_{-1}\rep  FL_{(0)}$ and the orthogonal complement represents $\Delta_{-2}(\ell)$ and $\Delta_{-2}(L)$.
\end{enumerate}
Condition \eqref{UL3.4} (4) requires that $FI_{-1}\rep  FL_{\le -1}\bot\dgf{1}$ and that the orthogonal complement represents $1$ or $\Delta$. Note that $FI_{-1}$ is always contained in $U_1=\bH^m$ or $U_2=\bH^{m-1}\bot \dgf{1,\,-\Delta}$, and $FL_1$ always contains $U_1$ and $U_2$. Moreover, $U_1\bot\dgf{1}\cong U_2\bot\dgf{\Delta}$ and $U_2\bot\dgf{1}\cong U_1\bot \dgf{\Delta}$. So \eqref{UL3.4} (4) always holds when $i=-2$.

For $i=-1$, $\ell_{\le -1}=I_{-1}$, $\ell_{[-1]}=I_{-1}\bot I_0$, and
\[
\Delta_{-1}(\ell)=\begin{cases}
  \cO_F \quad & \text{ if } P_0\neq 0\,,\\
  2\cO_F\quad & \text{ if } P_0=0\neq P_1\,,\\
  0 \quad & \text{ otherwise}\,.
\end{cases}
\]Now \eqref{UL3.4} (1) and (2) can be stated as  the following two conditions:
\begin{enumerate}
  \item[(b)] $FI_{-1}\bot FI_0\rep  FL_{(1)}$ and the orthogonal complement represents $\Delta_{-1}(\ell)$ and $\Delta_{-1}(L)$.
  \item[(c)] If $ FL_{(1)}/^\bot FI_{-1}\bot FI_0\cong\bH$ and $P_0=0\neq P_1$, then $L$ has no proper unimodular component.
\end{enumerate}
Condition \eqref{UL3.4} (3) is implied by condition (a) above. If the difference between the dimensions of $ FL_{\le 0}\bot \dgf{2}$ and $FI_{-1}\bot FI_0$ is smaller than $3$,  then the difference must be 2 and $d( FL_{\le 0}\bot \dgf{2})/d(FI_{-1}\bot FI_0)$ is a uniformizer up to squares. So \eqref{UL3.4} (4) holds automatically by a discriminant comparison (see \cite[(63:11) and (63:21)]{OMeara00}).

For $i=0$, $\ell_{\le 0}=\ell_{[0]}=I_{-1}\bot I_0\bot P_0$, $\Delta_0(\ell)$ is 0 or $2\cO_F$ depending on whether $P_1=0$ or not, and $\Delta_0(L)$ is always contained in $2\cO_F$. Thus \eqref{UL3.4} (2) holds automatically, and \eqref{UL3.4} (3) implies \eqref{UL3.4} (4). The conditions to check are the following:
\begin{enumerate}
  \item[(d)] $FI_{-1}\bot FI_0\bot FP_0\rep  FL_{(2)}$ and the orthogonal complement represents $\Delta_{0}(\ell)$ and $\Delta_{0}(L)$.
  \item[(e)] $FI_{-1}\bot FI_0\bot FP_0\rep  FL_{(1)}\bot \dgf{1}$ and the orthogonal complement represents $1$ or $\Delta$.
\end{enumerate}

For $i\ge 1$, we have $\ell_{\le i}=\ell_{[i]}=\ell$ and $\Delta_i(\ell)=0$. So \eqref{UL3.4} (2) is trivially verified, and it remains to check the following three conditions:
\begin{enumerate}
  \item[(f)] $F\ell\rep  FL_{(3)}$.
  \item[(g)] For every $i\ge 1$, $ FL_{(i+2)}/^\bot F\ell$ represents $\Delta_i(L)$.
  \item[(h)] $F\ell\rep  FL_{(2)}\bot \dgf{2}$ and the orthogonal complement represents $2$ or $2\Delta$.
\end{enumerate}

\begin{lemma}\label{UN6.4}
  Suppose $L_1$ is improper $2^{-1}$-modular of dimension $\ge k+1$.

  Then condition $\eqref{UL6.3}\;(a)$ holds.
\end{lemma}
\begin{proof}
  We may assume $\dim L_1=k+1$. Then $FL_1$ is isomorphic to $\bH^{m+1}$ or $\bH^m\bot \dgf{1,\,-\Delta}$. The space $FI_{-1}$ is always contained in $\bH^m$ or $\bH^{m-1}\bot\dgf{1,\,-\Delta}$. Thus $FL_1$ represents $FI_{-1}$, and the orthogonal complement contains $\bH$ or $\dgf{1,\,-\Delta}$, so that it always represents $\cO_F$. Since $\Delta_{-2}(\ell)$ is  either 0 or $\cO_F$ and similarly for $\Delta_{-2}(L)$, the result follows.
\end{proof}

\begin{prop}\label{UL6.5}
  Suppose $L_1$ is improper $2^{-1}$-modular of dimension $\ge k+1$.

  Then $L$ is $k$-universal if and only if one of the following holds:

  \begin{enumerate}
    \item $\dim L_1\ge k+3$;
    \item $\dim L_1=k+1$, $\dim L_2\ge 2$ and $2\cO_F\subseteq \fn(L_2)$;
    \item $\dim L_1=k+1$, $\dim L_2=1$, $\frs(L_2)=\fn(L_2)=\cO_F$ and $4\cO_F\subseteq \fn(L_3)$;
    \item $\dim L_1=k+1$, $\dim L_2=1$, $\frs(L_2)=\fn(L_2)=2\cO_F$ and $8\cO_F\subseteq \fn(L_3)$.
  \end{enumerate}
\end{prop}
\begin{proof}
  If $\dim L_1\ge k+3$, then all the conditions (a)--(h) in \eqref{UL6.3} hold for dimension reasons. So let us assume $\dim L_1=k+1$. Then $FL_1$ is $\bH^{m+1}$ or $\bH^m\bot\dgf{1,\,-\Delta}$.  But $\bH^m\bot\dgf{1,\,-\Delta}$ does not represent $U_3=\bH^{m-1}\bot \dgf{2,\,-2\Delta}$, and $\bH^{m+1}/^\bot U_3\cong \dgf{2,\,-2\Delta}$. Since we may choose $\ell$ such that $FI_{-1}\bot FI_0=U_3$ and $\dim P_0=1$ (so that $\Delta_{-1}(\ell)=\cO_F$), condition \eqref{UL6.3} (b) forces $L_1\subset L_{(1)}$, i.e., $2\cO_F\subseteq \fn(L_2)$. If $\dim L_2\ge 2$, then $\dim L_{(1)}\ge k+3$ and all the conditions (a)--(h) in \eqref{UL6.3} are satisfied.

  In the rest of the proof we further assume that $\dim L_2=1$. We can check  conditions \eqref{UL6.3} (b), (c), (e) and (h)  by a dimension argument.

  First assume $\fn(L_2)=\cO_F$. Then $FL_1\bot FL_2$ has the form $\bH^m\bot\dgf{\eta}$ for some $\eta\in \cO_F^*$. The space $FI_{-1}\bot FI_0\bot FP_0$ can be $\bH^m\bot\dgf{\veps}$ for all $\veps\in \cO_F^*$. Therefore, if \eqref{UL6.3} (d) holds for all $\ell$ we must have $L_1\bot L_2\subset L_{(2)}$. When this condition holds we have $\dim L_{(2)}\ge k+3$ and hence \eqref{UL6.3} (d), (f) and (g) are all verified.

  Finally, suppose $\fn(L_2)=2\cO_F$. Then either $\dim  FL_{(2)}\ge \dim (FI_{-1}\bot FI_0\bot FP_0)+3$, or $FL[\![\le 2]]\!]=FL_1\bot FL_2$, $\dim FI_{-1}\bot FI_0\bot FP_0=k$ and $d( FL_{(2)})/d(FI_{-1}\bot FI_0\bot FP_0)$ is a uniformizer up to squares. Hence \eqref{UL6.3} (d) holds by \cite[(63:11), (63:21)]{OMeara00} and Proposition\;\ref{UL2.2}. If $L_{(2)}=L_1\bot L_2$, then $ FL_{(2)}$ has the form $\bH^{m+1}\bot\dgf{2\eta}$ or $\bH^m\bot \dgf{1,\,-\Delta}\bot \dgf{2\eta}$, where $\eta\in \cO_F^*$. But $F\ell$ can be $\bH^m\bot\dgf{\veps}$ or $\bH^{m-1}\bot \dgf{1,\,-\Delta}\bot\dgf{2\veps}$ for all $\veps\in \cO_F^*$. Since
  $\bH^{m+1}\bot\dgf{2\eta}$ does not represent $\bH^{m-1}\bot\dgf{1,\,-\Delta}\bot \dgf{2\Delta\eta}$ and $\bH^m\bot \dgf{1,\,-\Delta}\bot \dgf{2\eta}$ does not represent $\bH^m\bot\dgf{2\Delta\eta}$, we see that \eqref{UL6.3} (f) holds if and only if $L_1\bot L_2\subset L_{(3)}$, i.e., $8\cO_F\subseteq \fn(L_3)$. When this condition holds, $\dim L_{(3)}\ge k+3$ and thus \eqref{UL6.3} (g) follows for dimension reasons. This completes the proof of the proposition.
\end{proof}

\begin{lemma}\label{UL6.6}
  Suppose that $L_1$ is improper $2^{-1}$-modular of dimension $k-1$ and that $L_2$ is proper unimodular of dimension $\ge 2$.

\begin{enumerate}
  \item Condition $\eqref{UL6.3}\;(a)$ holds.
  \item The following statements are equivalent:

  \begin{enumerate}
    \item Conditions $\eqref{UL6.3}\; (b),\,(c),\,(e)$ and $(h)$ hold for all dominant $k$-dimensional lattices $\ell$.
    \item Conditions $\eqref{UL6.3}\; (b),\,(c)$ and $(e)$ hold for all dominant $k$-dimensional lattices $\ell$.
    \item $\dim L_{(1)}\ge k+2$, i.e., either $\dim L_2\ge 3$, or $\dim L_2=2$ and $\frs(L_3)=\fn(L_3)=2\cO_F$.
  \end{enumerate}
\end{enumerate}
\end{lemma}
\begin{proof}
  (1) The space $ FL_{(0)}=FL_1\bot FL_2$ has the form $\bH^m\bot\dgf{\eta_1,\cdots,\eta_r}$ with $r=\dim L_2\ge 2$ and $\eta_i\in\cO_F^*$. As
  \[
  \bH\bot\dgf{\eta_1}\cong \dgf{1,\,-\Delta}\bot \dgf{\Delta\eta_1}
  \]and $FI_{-1}$ is contained in $U_1= \bH^m$ or $U_2=\bH^{m-1}\bot\dgf{1,\,-\Delta}$, we have always $FI_{-1}\subseteq  FL_{(0)}$ and the orthogonal complement represents $\cO_F$.

  (2) For dimensional reasons we have (c)$\Rightarrow$(a). Let us prove (b)$\Rightarrow$(c) by contradiction. So we assume $\dim L_2=2$ and $L_{(1)}=L_1\bot L_2$. Then
  $ FL_{(1)}=FL_1\bot FL_2=\bH^m\bot \dgf{\eta_1,\,\eta_2}$ with $\eta_1,\,\eta_2\in \cO_F^*$ and $-\eta_1\eta_2=d_{\pm}(FL_1\bot FL_2)$. As $\ell$ varies over all dominant lattices of dimension $k$, it is possible to have $FI_{-1}\bot FI_0=\bH^m$ and $\dim P_1=1$. Therefore, condition \eqref{UL6.3} (c) forces $-\eta_1\eta_2\notin F^{*2}$. On the other hand, in the case with $FI_{-1}\bot FI_0=\bH^{m-1}\bot \dgf{2,\,-2\Delta}$, \eqref{UL6.3} (b) implies that $-\eta_1\eta_2\notin \Delta F^{*2}$.

  If $V:=\dgf{\eta_1,\,\eta_2,\,1}$ is anisotropic, then there exists $\veps\in \cO_F^*$ such that $\veps\notin Q(V)$ by Lemma\;\ref{UL2.4}. But $FI_{-1}\bot FI_0\bot FP_0$ can be $\bH^m\bot\dgf{\veps}$. So \eqref{UL6.3} (e) fails for such an $\ell$. If $V$ is isotropic, then $ FL_{(1)}\bot \dgf{1}\cong \bH^m\bot\bH\bot \dgf{-\eta_1\eta_2}$. Since $-\eta_1\eta_2\in \cO_F^*\setminus (\cO_F^{*2}\cup \Delta\cO_F^{*2})$, by Lemma\;\ref{UL2.5} we can find $\veps\in \cO_F^*$ such that the binary space $\dgf{-\eta_1\eta_2,\,-\veps}$ represents neither 1 nor $\Delta$. Then in the case $FI_{-1}\bot FI_0\bot FP_0=\bH^m\bot\dgf{\veps}$, the orthogonal complement of
  $FI_{-1}\bot FI_0\bot FP_0$ in $ FL_{(1)}\bot \dgf{1}$ is $\dgf{-\eta_1\eta_2,\,-\veps}$. This contradicts \eqref{UL6.3} (e).
\end{proof}

\begin{prop}\label{UL6.7}
  Suppose that $L_1$ is improper $2^{-1}$-modular of dimension $k-1$ and that $L_2$ is proper unimodular of dimension $\ge 2$.

  Then $L$ is $k$-universal if and only if one of the following holds:

  \begin{enumerate}
    \item $\dim L_2\ge 4$;
    \item $\dim L_2=3$ and $4\cO_F\subseteq \fn(L_3)$;
    \item $\dim L_2=2$, $\dim L_3\ge 2$ and $\frs(L_3)=\fn(L_3)=2\cO_F$;
    \item $\dim L_2=2$, $\dim L_3=1$, $\frs(L_3)=\fn(L_3)=2\cO_F$ and $8\cO_F\subseteq \fn(L_4)$.
  \end{enumerate}
\end{prop}
\begin{proof}
  We may assume that $L$ satisfies the equivalent conditions in Lemma\;\ref{UL6.6} (2). We only need to check conditions \eqref{UL6.3} (d), (f) and (g).

  If $L$ satisfies condition (1), (2) or (3) of the proposition, then \eqref{UL6.3} (d), (f) and (g) hold for dimensional reasons. If $L$ satisfies (4), then $d( FL_{(1)})$ is a uniformizer up to squares and we can check \eqref{UL6.3} (d) by using \cite[(63:11) and (63:21)]{OMeara00}. Since $\dim L_{(3)}\ge k+3$, \eqref{UL6.3} (f) and (g) hold for dimensional reasons. We have thus proved the sufficiency part of the proposition.

  For the necessity, let us suppose that \eqref{UL6.3} (d), (f) and (g) hold for all dominant $k$-dimensional lattices $\ell$.

  First assume $\dim L_2=3$. We want to show $L_1\bot L_2\subset L_{(2)}$. Indeed, $FL_1\bot FL_2=\bH^m\bot V$ with $V=\dgf{\eta_1,\eta_2,\,\eta_3}$ for some $\eta_i\in\cO_F^*$. If $V$ is anisotropic, there exists $\veps\in \cO_F^*$ such that $\veps\notin Q(V)$, by Lemma\;\ref{UL2.4}. But it is possible to have $FI_{-1}\bot FI_0\bot FP_0=\bH^m\bot\dgf{\veps}$. So by \eqref{UL6.3} (d), if $L_1\bot L_2=L_{(2)}$, $V$ must be isotropic and hence $ FL_{(2)}=\bH^{m+1}\bot\dgf{\eta}$ with $\eta=-\eta_1\eta_2\eta_3$. But then in the case
  \[
  FI_{-1}\bot FI_0\bot FP_0=\bH^{m-1}\bot \dgf{2,\,-2\Delta}\bot \dgf{\Delta\eta}
  \] the representation relation $FI_{-1}\bot FI_0\bot FP_0\rep  FL_{(2)}$ is not true. So we must have $L_1\bot L_2\subset L_{(2)}$ as desired.

  Now let us suppose $\dim L_2=2$ and $\frs(L_3)=\fn(L_3)=\cO_F$. We want to show that either $\dim L_3\ge 2$ or $8\cO_F\subseteq \fn(L_4)$. Suppose that $\dim L_3=1$. Then $FL_1\bot FL_2\bot FL_3=\bH^m\bot \dgf{\eta_1,\,\eta_2}\bot \dgf{2\eta_3}$ for some $\eta_i\in\cO_F^*$. Put $V=\dgf{\eta_1,\,\eta_2}\bot \dgf{2\eta_3}$. If $V$ is anisotropic, then there exists $\veps\in \cO_F^*$ such that $2\veps\notin Q(V)$. Thus, if $L_1\bot L_2\bot L_3=L_{(3)}$, \eqref{UL6.3} (f) fails when $F\ell=\bH^m\bot\dgf{2\veps}$. If $V$ is isotropic, then $FL_1\bot FL_2\bot FL_3=\bH^{m+1}\bot \dgf{2\eta}$ with $\eta=-\eta_1\eta_2\eta_3$. If $L_1\bot L_2\bot L_3=L_{(3)}$, then \eqref{UL6.3} (f) fails when $F\ell=\bH^{m-1}\bot\dgf{1,\,-\Delta}\bot \dgf{2\Delta\eta}$. So in any case we find that $L_{(3)}$ must be strictly larger than $L_1\bot L_2\bot L_3$, namely, $8\cO_F\subseteq \fn(L_4)$. This completes the proof of the proposition.
\end{proof}

\begin{lemma}\label{UL6.8}
  Suppose that $L_1$ is improper $2^{-1}$-modular of dimension $k-1$, that $L_2$ is unimodular of dimension $1$ and that $L_3$ is proper $2$-modular.

  Then conditions $\eqref{UL6.3}\;(a),\,(b),\,(c)$ and $(e)$ hold.
\end{lemma}
\begin{proof}
  Conditions \eqref{UL6.3} (b), (c) and (e) hold by a discriminant or dimension argument. We have $ FL_{(0)}=FL_1\bot FL_2=\bH^m\bot\dgf{\eta}\cong \bH^{m-1}\bot \dgf{1,\,-\Delta}\bot\dgf{\Delta\eta}$ with $\eta\in \cO_F^*$. The space $FI_{-1}$ is always contained in $U_1=\bH^m$ or $U_2=\bH^{m-1}\bot\dgf{1,\,-\Delta}$. Noticing that $\Delta_{-2}(\ell)$ is 0 or $\cO_F$ and similarly for $\Delta_{-2}(L)$, as in the proof of Lemma\;\ref{UL6.6} (1), we find that \eqref{UL6.3} (a) is satisfied.
\end{proof}

\begin{prop}\label{UL6.9}
 Suppose that $L_1$ is improper $2^{-1}$-modular of dimension $k-1$, that $L_2$ is unimodular of dimension $1$ and that $L_3$ is proper $2$-modular.

  Then $L$ is $k$-universal if and only if one of the following holds:

  \begin{enumerate}
    \item $\dim L_3\ge 3$;
    \item $\dim L_3=2$ and $\frs(L_4)=\fn(L_4)=4\cO_F$;
    \item $\dim L_3=1$, $\dim L_4\ge 2$ and $\frs(L_4)=\fn(L_4)=4\cO_F$;
    \item $\dim L_3=\dim L_4=1$, $\frs(L_4)=\fn(L_4)=4\cO_F$ and $\frs(L_5)=\fn(L_5)=8\cO_F$.
  \end{enumerate}
\end{prop}
\begin{proof}
  Thanks to Lemma\;\ref{UL6.8}, it remains to check conditions (d), (f), (g) and (h) in \eqref{UL6.3}. If $\dim L_3\ge 3$, all these conditions hold by a dimension count.

  Let us assume $\dim L_3=2$. If $\frs(L_4)=\fn(L_4)=4\cO_F$, then $\dim L_{(2)}\ge k+3$, so we are done again by a dimension argument. Conversely, we claim that if \eqref{UL6.3} (d) holds for all dominant $k$-dimensional lattices $\ell$, then $L_1\bot L_2\bot L_3\subset L_{(2)}$. Indeed, we have $FL_1\bot FL_2\bot FL_3=\bH^m\bot\dgf{\eta_1}\bot \dgf{2\veps_1,\,2\veps_2}$ for some $\eta,\,\veps_1,\,\veps_2\in \cO_F^*$. Put $V=\dgf{\eta_1,\,2\veps_1,\,2\veps_2}$ and $U=FI_{-1}\bot FI_0\bot FP_0$. If $V$ is isotropic, then $V\cong \bH\bot \dgf{\veps}$ with $\veps=-\eta_1\veps_1\veps_2$. In the case $U=\bH^{m-1}\bot \dgf{2,\,-2\Delta}\bot \dgf{\Delta\veps}$, $FL_1\bot FL_2\bot FL_3$ does not represent $U$. If $V$ is anisotropic, then it does not represent $\veps=-\eta_1\veps_1\veps_2$. Thus, in the case $U=\bH^m\bot \dgf{\veps}$ the space $FL_1\bot FL_2\bot FL_3$ does not represent $U$. Our claim is thus proved.

  Next let us assume $\dim L_3=1$. Then $FL_1\bot FL_2\bot FL_3=\bH^m\bot \dgf{\eta_1}\bot \dgf{2\veps_1}$ for some $\eta_1,\,\veps_1\in \cO_F^*$. By \cite[(63:11)]{OMeara00}, there exists $\veps\in \cO_F^*$ with $\veps\notin Q(\dgf{\eta_1,\,2\veps_1})$. So the space $U=FI_{-1}\bot FI_0\bot FP_0$ is not represented by $FL_1\bot FL_2\bot FL_3$ if $U=\bH^m\bot\dgf{\veps}$. So by \eqref{UL6.3} (d) we may assume $L_1\bot L_2\bot L_3\subset L_{(2)}$, i.e., $\frs(L_4)=\fn(L_4)=4\cO_F$.

  If $\dim L_4\ge 2$, then $\dim L_{(2)}\ge k+3$ and therefore \eqref{UL6.3} (d), (f), (g) and (h) all hold for dimensional reasons. So let us assume $\dim L_4=1$. Thus $ FL_{(2)}=\bH^m\bot \dgf{\eta_1}\bot \dgf{2\veps_1}\bot \dgf{\eta_2}$ with $\eta_1,\,\eta_2,\,\veps_1\in \cO_F^*$. Then \eqref{UL6.3} (d) is verified by a dimension count or a discriminant comparison (Proposition\;\ref{UL2.2} and \cite[(63:11) and (63:21)]{OMeara00}).

  If $8\cO_F\subseteq \fn(L_5)$, then conditions \eqref{UL6.3} (f)--(h) are all verified by a dimension count. Conversely, it remains to prove that if \eqref{UL6.3} (f) holds for all $\ell$, then we must have $8\cO_F\subseteq \fn(L_5)$. Assume the contrary. Then $ FL_{(3)}=FL_1\bot FL_2\bot FL_3\bot FL_4=\bH^m\bot V$ with $V=\dgf{\eta_1,\,\eta_2,\,2\veps_1}$ with $\eta_1,\,\eta_2,\,\veps_1\in \cO_F^*$. If $V$ is isotropic, then $ FL_{(3)}\cong \bH^{m+1}\bot\dgf{2\eta}$ with $\eta=-\eta_1\eta_2\veps_1$. But then $F\ell$ is not represented by $ FL_{(3)}$ when $F\ell=\bH^{m-1}\bot \dgf{1,\,-\Delta}\bot \dgf{2\Delta\eta}$. If $V$ is anisotropic, then there exists $\veps\in \cO_F^*$ such that $2\veps\notin Q(V)$. In this case $F\ell$ is not represented by $ FL_{(3)}$ if $F\ell=\bH^m\bot\dgf{2\veps}$. This completes the proof.
\end{proof}

Summarizing Propositions\;\ref{UL6.5}, \ref{UL6.7} and \ref{UL6.9} yields the following result:

\begin{thm}\label{UL6.10}
  Let $F$ be an unramified extension of $\Q_2$ and let $L=L_1\bot \cdots\bot L_t$ be a Jordan splitting of an $\cO_F$-lattice.
 Let $k\ge 3$ be an odd integer.

  Then $L$ is $k$-universal if and only if $\fn(L_1)=\cO_F$, $\frs(L_1)=2^{-1}\cO_F$  and one of the following conditions holds:

\begin{enumerate}
  \item $\dim L_1\ge k+3$.

  \item $\dim L_1=k+1$ and one of the following cases happens:

   \begin{enumerate}
    \item $\dim L_2\ge 2$ and $2\cO_F\subseteq \fn(L_2)$;
    \item $\dim L_2=1$, $\frs(L_2)=\fn(L_2)=\cO_F$ and $4\cO_F\subseteq \fn(L_3)$;
    \item $\dim L_2=1$, $\frs(L_2)=\fn(L_2)=2\cO_F$ and $8\cO_F\subseteq \fn(L_3)$.
  \end{enumerate}
  \item $\dim L_1=k-1$, $\frs(L_2)=\fn(L_2)=\cO_F$ and one of the following cases happens:

  \begin{enumerate}
    \item $\dim L_2\ge 4$;
    \item $\dim L_2=3$ and $4\cO_F\subseteq \fn(L_3)$;
    \item $\dim L_2=2$, $\dim L_3\ge 2$ and $\frs(L_3)=\fn(L_3)=2\cO_F$;
    \item $\dim L_2=2$, $\dim L_3=1$, $\frs(L_3)=\fn(L_3)=2\cO_F$ and $8\cO_F\subseteq \fn(L_4)$.
  \end{enumerate}

  \item $\dim L_1=k-1$,  $\dim L_2=1$, $\frs(L_2)=\fn(L_2)=\cO_F$, $\frs(L_3)=\fn(L_3)=2\cO_F$ and one of the following cases happens:

  \begin{enumerate}
    \item $\dim L_3\ge 3$;
    \item $\dim L_3=2$ and $\frs(L_4)=\fn(L_4)=4\cO_F$;
    \item $\dim L_3=1$, $\dim L_4\ge 2$ and $\frs(L_4)=\fn(L_4)=4\cO_F$;
    \item $\dim L_3=\dim L_4=1$, $\frs(L_4)=\fn(L_4)=4\cO_F$ and $\frs(L_5)=\fn(L_5)=8\cO_F$.
  \end{enumerate}
\end{enumerate}
\end{thm}

\begin{prop}\label{UL6.2}
  Suppose $L$ is classic. Then the following are equivalent:

  \begin{enumerate}
    \item Every classic basic lattice of dimension $k$ has a lower type than $L$.
    \item $L_1$ is proper unimodular of dimension $\ge k$, and if $\dim L_1=k$, then $\frs(L_2)=\fn(L_2)=2\cO_F$.
  \end{enumerate}
\end{prop}
\begin{proof}
  First, condition \eqref{UL3.2} (1) holds for all $\ell$ if and only if $\dim  L_{\le 0}\ge k$, i.e., $L_1$ is unimodular of dimension $\ge k$. If $i<0$, there is no need to check \eqref{UL3.2} (2) and (4). If $i<-1$, then \eqref{UL3.2} (3) holds automatically. For $i=-1$, \eqref{UL3.2} (3) requires that if $\ell$ has a proper unimodular component, then so does $L$. This being true for all $\ell$, we see that $L$ must have a proper unimodular component. So we may assume $L_1$ is proper unimodular of dimension $\ge k$.

  Next consider the case $i=0$. If $\dim\ell_{\le 0}=\dim  L_{\le 0}\ge k$, then $\ell=\ell_{\le 0}=I_0\bot P_0$ and $\dim  L_{\le 0}=k$. Since $k$ is odd and $\dim I_0$ is even, we have $P_0\neq 0$. So $\ell$ has no proper 2-modular component but has a proper unimodular component. Also, $\ord(\fd_0(\ell)\fd_0(L))$ is always even. So conditions \eqref{UL3.2} (2) and (3) hold for $i=0$. Suppose that $\dim \ell_{\le 0}=\dim  L_{\le 0}-1$. Then either $\dim \ell_{\le 0}=k$ or $\dim \ell_{\le 0}=k-1$. In the former case $I_1=P_1=0$, so \eqref{UL3.2} (4) is satisfied. In the latter case, condition \eqref{UL3.2} (4) means that if $\ell$ has a proper $2$-modular component, then so does $L$. Therefore, \eqref{UL3.2} (4) holds for all $\ell$ if and only if when $\dim L_1=k$ we have $\frs(L_2)=\fn(L_2)=2\cO_F$.

  In the rest of the proof we assume that $L$ satisfies the conditions in (2) and we want to check that conditions \eqref{UL3.2} (2)--(4) are satisfied for all $i\ge 1$.

  We now have $\dim  L_{\le 1}\ge k+1>\dim \ell$. So we only need to consider \eqref{UL3.2} (4). Suppose that for some $i\ge 1$ we have $\dim L_{\le i}=k+1$. Then we get $L_{\le i}= L_{\le 1}$. So if $i>1$, then $L$ has no $2^i$-modular component. In this case \eqref{UL3.2} (4) clearly holds. For $i=1$, we conclude that $\dim  L_{\le 1}=k+1$. If $\dim L_1=k+1$, this implies that $L$ has no 2-modular component, and hence \eqref{UL3.2} (4) is verified. Otherwise we have $\dim L_1=k$ and $L_2$ is 2-modular of dimension 1. In this case $\ord(\fd_1(L))\equiv 1\pmod{2}$. On the other hand, $\ord(\fd_1(\ell))\equiv \dim P_1\pmod{2}$. Therefore, if $\ord(\fd_1(\ell)\fd_1(L))\equiv 2\pmod{2}$, we have $P_1\neq 0$ and thus \eqref{UL3.2} (4) holds. This completes the proof.
\end{proof}

\newpara\label{UL6.11} Suppose that $L$ satisfies the equivalent conditions in Proposition\;\ref{UL6.2}, i.e., $L_1$ is proper unimodular of dimension $\ge k$ and if $\dim L_1=k$, then $L_2$ is proper $2$-modular. In particular, we have $\dim L_{(1)}\ge k+1$ and $\dim  L_{\le 1}\ge k+1$.

Let $\ell=I_0\bot P_0\bot I_1\bot P_1$ be an arbitrary classic basic lattice of dimension $k=2m+1\ge 3$.

If $i<-2$, then $\ell_{[i]}=\ell_{\le i}=0$, $\Delta_i(\ell)=\Delta_i(L)=0$. In this case all the conditions in Theorem\;\ref{UL3.4} are trivially verified.

For $i=-2$, we have $\ell_{[-2]}=\ell_{\le -2}=0$, $\Delta_{-2}(\ell)$ is 0 or $\cO_F$ depending on whether $P_0=0$ or not, and $\Delta_{-2}(L)=\cO_F$. It is easily seen that all the conditions in Theorem\;\ref{UL3.4} hold in this case.

For $i=-1$, we have $\ell_{[-1]}=I_0$, $\ell_{\le -1}=0$,
\[
 \Delta_{-1}(L)=\cO_F\quad\text{and}\quad \Delta_{-1}(\ell)=\begin{cases}
  \cO_F\quad & \text{ if } P_0\neq 0\,,\\
  2\cO_F\quad & \text{ if } P_0=0\neq P_1\,,\\
  0 \quad & \text{ otherwise}\,.
\end{cases}
\]Condition \eqref{UL3.4} (1) means that
\begin{enumerate}
  \item[(a)] $FI_0\rep  FL_{(1)}$ and the orthogonal complement represents $\Delta_{-1}(\ell)$ and $\Delta_{-1}(L)$.
\end{enumerate}
Condition \eqref{UL3.4} (2) is equivalent to the following:
\begin{enumerate}
  \item[(b)] If $ FL_{(1)}/^\bot FI_0\cong \bH$, then either $P_0\neq 0$ or $P_0=P_1=0$.
\end{enumerate}
Trivially \eqref{UL3.4} (3) holds. Condition \eqref{UL3.4} (4) requires that $FI_0\rep FL_1\bot \dgf{2}$ and that the orthogonal complement represents 2 or $2\Delta$. This condition always holds for dimension or discriminant reasons.

Next consider the case $i=0$. We have $\ell_{[0]}=I_0\bot P_0\bot I_1$, $\ell_{\le 0}=I_0\bot P_0$, $\Delta_0(\ell)$ is 0 or $2\cO_F$ depending on whether $P_1=0$ or not, and $\Delta_0(L)\subseteq 2\cO_F$. Thus \eqref{UL3.4} (2) holds automatically, \eqref{UL3.4} (1) says:
\begin{enumerate}
  \item[(c)] $FI_0\bot FP_0\bot FI_1\rep  FL_{(2)}$ and the orthogonal complement represents $\Delta_{0}(\ell)$ and $\Delta_{0}(L)$.
\end{enumerate}Conditions \eqref{UL3.4} (3) and (4) are translated into the following two conditions:
\begin{enumerate}
  \item[(d)] $FI_0\bot FP_0\rep  FL_{(1)}\bot \dgf{1}$ and that the orthogonal complement represents 1 or $\Delta$.
  \item[(e)] $FI_0\bot FP_0\bot FI_1\rep  FL_{\le 1}\bot \dgf{1}$ and that the orthogonal complement represents 1 or $\Delta$.
\end{enumerate}

For $i\ge 1$ we have $\ell_{[i]}=\ell_{\le i}=\ell$ and $\Delta_i(\ell)=0$. So \eqref{UL3.4} (2) holds trivially and \eqref{UL3.4} (4)  is a consequence of \eqref{UL3.4} (3). We may restate conditions \eqref{UL3.4} (1) and (3) as the following conditions:
\begin{enumerate}
  \item[(f)] $F\ell\rep  FL_{(3)}$;
  \item[(g)] for every $i\ge 1$, $ FL_{(i+2)}/^\bot F\ell$ represents $\Delta_i(L)$;
  \item[(h)] $F\ell\rep  FL_{(2)}\bot \dgf{2}$ and the orthogonal complement represents 2 or $2\Delta$.
\end{enumerate}

\begin{prop}\label{UL6.12}
Suppose $L_1$ is proper unimodular of dimension $\ge k+2$.

Then $L$ is classic $k$-universal if and only if either $\dim L_1\ge k+3$, or $\dim L_1=k+2$ and $4\cO_F\subseteq \fn(L_2)$.
\end{prop}
\begin{proof}
  Under the assumption $\dim L_1\ge k+2$, conditions \eqref{UL6.11} (a), (b), (d), (e) and (h) follow from a dimension count. It remains to show that conditions \eqref{UL6.11} (c), (f) and (g) hold for all classic basic lattices $\ell$ of dimension $k$ if and only if $L$ has the property described in the proposition. The ``if'' part is immediate for dimensional reasons.

  Now let us suppose $\dim L_1=k+2$. Then $FL_1=\bH^m\bot V$ with $V=\dgf{\eta_1,\,\eta_2,\,\eta_3}$ for some $\eta_i\in\cO_F^*$. It is sufficient to show that if \eqref{UL6.11} (c) holds for all $\ell$, then we must have $L_1\subset L_{(2)}$.

  Assume the contrary. If $V$ is isotropic, then $ FL_{(2)}=\bH^{m+1}\bot \dgf{\eta}$ with $\eta=-\eta_1\eta_2\eta_3\in \cO_F^*$. In this case \eqref{UL6.11} (c) fails when $FI_0\bot FP_0\bot FI_1=\bH^{m-1}\bot\dgf{2,\,-2\Delta}\bot \dgf{\Delta\eta}$. If $V$ is anisotropic, then by Lemma\;\ref{UL2.4} there exists $\veps\in \cO_F^*$ such that $\veps\notin Q(V)$. Then \eqref{UL6.11} (c) fails when $FI_0\bot FP_0\bot FI_1=\bH^m\bot\dgf{\veps}$.
\end{proof}

\begin{lemma}\label{UL6.13}
Suppose $L_1$ is proper unimodular of dimension $k+1$. Then the following are equivalent:

\begin{enumerate}
  \item Conditions $\eqref{UL6.11}\,(a),\,(b),\,(d),\,(e)$ and $(h)$ hold for all classic basic lattices $\ell$ of dimension $k$.
  \item Conditions $\eqref{UL6.11}\,(a),\,(b)$ and $(d)$ hold for all classic basic lattices $\ell$ of dimension $k$.
  \item $L_1\subset L_{(1)}$, i.e., $\frs(L_2)=\fn(L_2)=2\cO_F$.
\end{enumerate}
\end{lemma}
\begin{proof}
  By dimension count we have (3)$\Rightarrow$(1). Let us prove (2)$\Rightarrow$(3) by contradiction.

  Suppose that $L_1=L_{(1)}$. Then $ FL_{(1)}=FL_1=\bH^{m-1}\bot V$ with $V=\dgf{\eta_1,\cdots, \eta_4}$ for some $\eta_i\in \cO_F^*$. If \eqref{UL6.11} (a) holds for all $\ell$, $V$ must be isotropic and represent $\dgf{2,\,-2\Delta}$. Thus $V\cong \bH\bot U$ with $U\cong \dgf{2,\,-2\Delta}$ or $-d(U)\notin \Delta F^{*2}$. If $U\cong \dgf{2,\,-2\Delta}$, then in the case $FI_0\bot FP_0=\bH^m\bot\dgf{1}$ condition \eqref{UL6.11} (d) fails. If $-d(U)\in F^{*2}$, then $ FL_{(1)}=\bH^{m+1}$ and in the case with $FI_0=\bH^m$ and $\dim P_1=1$, condition \eqref{UL6.11} (b) fails. If $-d(U)\notin F^{*2}\cup \Delta F^{*2}$, then $c:=-d(U)\in \cO_F^*\setminus (\cO_F^{*2}\cup \Delta \cO_F^{*2})$, and by Corollary\;\ref{UL2.6} we may write $U=\dgf{\eta,\,-\eta c}$ with $\eta\in \cO_F^*$. If $U\bot \dgf{1}$ is isotropic, then $ FL_{(1)}\bot \dgf{1}=\bH^{m+1}\bot\dgf{c}$. Choosing $\veps\in \cO_F^*$ such that $\dgf{c,\,-\veps}$ represents neither 1 nor $\Delta$ (which is possible by Lemma\;\ref{UL2.5}), we find that \eqref{UL6.11} (d) fails when $FI_0\bot FP_0=\bH^m\bot \dgf{\veps}$, because in that case $(FL_{(1)}\bot \dgf{1})/^\bot (FI_0\bot FP_0)$ is isomorphic to $\dgf{c,\,-\veps}$. If $U\bot \dgf{1}$ is anisotropic, then there exists $\veps\in \cO_F^*$ such that $\veps\notin Q(U\bot \dgf{1})$ by Lemma\;\ref{UL2.4}. Then \eqref{UL6.11} (d) fails when $FI_0\bot FP_0=\bH^m\bot\dgf{\veps}$.

  The above discussions together prove (2)$\Rightarrow$(3), so the lemma is proved.
  \end{proof}

\begin{prop}\label{UL6.14}
Suppose $L_1$ is proper unimodular of dimension $k+1$.

Then $L$ is classic $k$-universal if and only if $\frs(L_2)=\fn(L_2)=2\cO_F$ and one of the following holds:

\begin{enumerate}
  \item $\dim L_2\ge 2$;
  \item $\dim L_2=1$ and $8\cO_F\subseteq \fn(L_3)$.
\end{enumerate}
\end{prop}
\begin{proof}
  By Lemma\;\ref{UL6.13}, we may assume that $\frs(L_2)=\fn(L_2)=2\cO_F$. We only need to check that conditions \eqref{UL6.11} (c), (f) and (g) hold for all classic basic lattices $\ell$ of dimension $k$ if and only if one of the two conditions in the proposition is satisfied.

  If $\dim L_2\ge 2$, then the result is immediate by a dimension count. Let us suppose $\dim L_2=1$. Then $FL_1\bot FL_2$ is a $(k+2)$-dimensional space whose discriminant is a uniformizer up to squares. Thus, \eqref{UL6.11} (c) is verified by a discriminant comparison. We will prove that \eqref{UL6.11} (f) and (g) hold for all $\ell$ if and only if $L_1\bot L_2\subset L_{(3)}$.

  The ``if'' part is easily seen by a dimension argument. Let us assume $L_1\bot L_2=L_{(3)}$ and show that in this case \eqref{UL6.11} (f) does not hold for all $\ell$. In fact we can write $ FL_{(3)}=FL_1\bot FL_2$ as $\bH^m\bot V$, where $V$ is a ternary space with $d(V)$ a uniformizer up to squares. If $V$ is isotropic, then $ FL_{(3)}=\bH^{m+1}\bot\dgf{2\eta}$ for some $\eta\in \cO_F^*$. Then \eqref{UL6.11} (f) fails when $F\ell=\bH^{m-1}\bot\dgf{1,\,-\Delta}\bot \dgf{2\Delta\eta}$. If $V$ is anisotropic we can find $\veps\in \cO_F^*$ such that $2\veps\notin Q(V)$ by Lemma\;\ref{UL2.4}. Thus \eqref{UL6.11} (f) fails when $F\ell=\bH^m\bot\dgf{2\veps}$. This completes the proof.
\end{proof}

\begin{prop}\label{UL6.15}
Suppose that $L_1$ is proper unimodular of dimension $k$ and that $L_2$ is proper $2$-modular.

Then $L$ is classic $k$-universal if and only if one of the following conditions holds:

\begin{enumerate}
  \item $\dim L_2\ge 3$;
  \item $\dim L_2=2$ and $\frs(L_3)=\fn(L_3)=4\cO_F$;
  \item $\dim L_2=1$, $\dim L_3\ge 2$ and $\frs(L_3)=\fn(L_3)=4\cO_F$;
  \item $\dim L_2=\dim L_3=1$, $\frs(L_3)=\fn(L_3)=4\cO_F$ and $\frs(L_4)=\fn(L_4)=8\cO_F$.
\end{enumerate}
\end{prop}
\begin{proof}
  If $\dim L_2\ge 3$, then $L_{(1)}$ has dimension $\ge k+3$ and all the conditions (a)--(h) in \eqref{UL6.11} are verified for dimension reasons. It suffices to discuss separately the two cases with $\dim L_2=2$ and $\dim L_2=1$ respectively.

\medskip

\noindent {\bf Case 1.} $\dim L_2=2$.

\medskip

In this case $\dim L_{(1)}=k+2$. Conditions \eqref{UL6.11} (a), (b), (d), (e) and (h) follow by a dimension count. It remains to check conditions \eqref{UL6.11} (c), (f) and (g).

We have $ FL_{(1)}=\bH^m\bot V$ with $\dim V=3$ and $d(V)\in \cO_F^*$. If $V$ is isotropic, then $V=\bH\bot\dgf{\eta}$ for some $\eta\in \cO_F^*$. In this event \eqref{UL6.11} (c) fails when $FI_0\bot FP_0\bot FI_1=\bH^{m-1}\bot \dgf{2,\,-2\Delta}\bot \dgf{\Delta\eta}$. If $V$ is anisotropic, there exists $\veps\in \cO_F^*$ such that $\veps\notin Q(V)$ by Lemma\;\ref{UL2.4}. Thus \eqref{UL6.11} (c) fails when $FI_0\bot FP_0\bot FI_1=\bH^m\bot\dgf{\veps}$.

So we see that if \eqref{UL6.11} (c) holds for all $k$-dimensional classic basic lattices $\ell$, we must have $L_1\bot L_2\subset L_{(2)}$, i.e., $\frs(L_3)=\fn(L_3)=4\cO_F$. Conversely, when $L_1\bot L_2\subset L_{(2)}$ we can check \eqref{UL6.11} (c), (f) and (g) by a dimension argument.

\medskip

\noindent {\bf Case 2.} $\dim L_2=1$.

\medskip

Now $L_{(1)}$ has dimension $k+1$ and $ FL_{(1)}=\bH^m\bot V$ for some binary space $V$ with $d(V)$ a uniformizer up to squares. Thus, conditions \eqref{UL6.11} (a), (b), (d) and (e) follow by a discriminant comparison. It remains to check conditions \eqref{UL6.11} (c), (f), (g) and (h).

By \cite[(63:11)]{OMeara00}, there exists $\veps\in \cO_F^*$ such that $\veps\notin Q(V)$. If $L_1\bot L_2=L_{(2)}$, then \eqref{UL6.11} (c) fails when $FI_0\bot FP_0\bot FI_1=\bH^m\bot \dgf{\veps}$. So we have $L_1\bot L_2\subset L_{(2)}$, i.e., $\frs(L_3)=\fn(L_3)=4\cO_F$.

If $\dim L_3\ge 2$, then $\dim L_{(2)}\ge k+3$. In this case \eqref{UL6.11} (c), (f), (g) and (h) are all satisfied for dimension reasons.

Let us assume $\dim L_3=1$. Then $ FL_{(2)}=FL_1\bot FL_2\bot FL_3=\bH^m\bot V\bot \dgf{\eta}$ for some $\eta\in \cO_F^*$. Condition \eqref{UL6.11} (c) follows from a discriminant comparison and \eqref{UL6.11} (h) holds for dimension reasons. If $L_1\bot L_2\bot L_3\subset L_{(3)}$, then \eqref{UL6.11} (f) and (g) also hold for dimension reasons.

Let us finally prove that if \eqref{UL6.11} (f) holds for all $\ell$, then $L_1\bot L_2\bot L_3\subset L_{(3)}$. Assume the contrary. If $V\bot \dgf{\eta}$ is isotropic, then $ FL_{(3)}=FL_1\bot FL_2\bot FL_3=\bH^{m+1}\bot \dgf{2\eta'}$ with $\eta'\in \cO_F^*$. Then \eqref{UL6.11} (f) fails when $F\ell=\bH^{m-1}\bot\dgf{1,\,-\Delta}\bot \dgf{2\Delta\eta'}$. If $V\bot \dgf{\eta}$ is anisotropic, there exists $\veps\in \cO_F^*$ such that $2\veps\notin Q(V\bot\dgf{\eta})$. Then \eqref{UL6.11} (f) fails when $F\ell=\bH^m\bot\dgf{2\veps}$. This proves the desired result.
\end{proof}

Summarizing Propositions\;\ref{UL6.12}, \ref{UL6.14} and \ref{UL6.15} we obtain the following theorem.

\begin{thm}\label{UL6.16}
 Let $F$ be an unramified extension of $\Q_2$ and let  $L=L_1\bot \cdots\bot L_t$ be a Jordan splitting of an $\cO_F$-lattice.
Let $k$ be an odd integer $\ge 3$.

  Then $L$ is classic $k$-universal if and only if $\frs(L_1)=\fn(L_1)=\cO_F$ and one of the following conditions holds:

\begin{enumerate}
  \item $\dim L_1\ge k+3$.
  \item $\dim L_1=k+2$,  and $4\cO_F\subseteq \fn(L_2)$.
  \item $\dim L_1=k+1$, $\frs(L_2)=\fn(L_2)=2\cO_F$ and one of the following cases happens:

\begin{enumerate}
  \item $\dim L_2\ge 2$;
  \item $\dim L_2=1$ and $8\cO_F\subseteq \fn(L_3)$.
\end{enumerate}
  \item $\dim L_1=k$, $\frs(L_2)=\fn(L_2)=2\cO_F$ and one of the following cases happens:

\begin{enumerate}
  \item $\dim L_2\ge 3$;
  \item $\dim L_2=2$ and $\frs(L_3)=\fn(L_3)=4\cO_F$;
  \item $\dim L_2=1$, $\dim L_3\ge 2$ and $\frs(L_3)=\fn(L_3)=4\cO_F$;
  \item $\dim L_2=\dim L_3=1$, $\frs(L_3)=\fn(L_3)=4\cO_F$ and $\frs(L_4)=\fn(L_4)=8\cO_F$.
\end{enumerate}
\end{enumerate}
\end{thm}

\noindent \emph{Acknowledgements.} We thank Prof.\,Fei Xu for helpful discussions. The authors were supported by a grant from the National Natural Science Foundation of China (no.\,12171223) and the Guangdong Basic and Applied Basic Research Foundation (no.\,2021A1515010396).

\addcontentsline{toc}{section}{\textbf{References}}

\bibliographystyle{alpha}

\bibliography{Univ2adic}

\begin{thebibliography}{{Ram}17}

\bibitem[Bel01]{Beli01thesis}
C.~N. Beli.
\newblock {\em Integral spinor norm groups over dyadic local fields and
  representations of quadratic lattices}.
\newblock ProQuest LLC, Ann Arbor, MI, 2001.
\newblock Thesis (Ph.D.)--The Ohio State University.

\bibitem[Bel06]{Beli06}
C.~N. Beli.
\newblock Representations of integral quadratic forms over dyadic local fields.
\newblock {\em Electron. Res. Announc. Amer. Math. Soc.}, 12:100--112, 2006.

\bibitem[Bel10]{Beli10TAMS}
C.~N. Beli.
\newblock A new approach to classification of integral quadratic forms over
  dyadic local fields.
\newblock {\em Trans. Amer. Math. Soc.}, 362(3):1599--1617, 2010.

\bibitem[Bel19]{Beli19}
C.~N. Beli.
\newblock Representations of quadratic lattices over dyadic local fields.
\newblock preprint available at arXiv:1905.04552, 2019.

\bibitem[Bel20]{Beli20}
C.~N. Beli.
\newblock Universal integral quadratic forms over dyadic local fields.
\newblock preprint available at arXiv:2008.10113, 2020.

\bibitem[Dic27a]{Dickson27BAMSpage63}
L.~E. Dickson.
\newblock Integers represented by positive ternary quadratic forms.
\newblock {\em Bull. Amer. Math. Soc.}, 33(1):63--70, 1927.

\bibitem[Dic27b]{Dickson27AJMpage39}
L.~E. Dickson.
\newblock Quaternary {Q}uadratic {F}orms {R}epresenting all {I}ntegers.
\newblock {\em Amer. J. Math.}, 49(1):39--56, 1927.

\bibitem[Dic29a]{Dickson29BAMS}
L.~E. Dickson.
\newblock The forms {$ax^2+by^2+cz^2$} which represent all integers.
\newblock {\em Bull. Amer. Math. Soc.}, 35(1):55--59, 1929.

\bibitem[Dic29b]{Dickson29TAMS}
L.~E. Dickson.
\newblock Universal quadratic forms.
\newblock {\em Trans. Amer. Math. Soc.}, 31(1):164--189, 1929.

\bibitem[{Dic}30]{Dickson30}
L.~E. {Dickson}.
\newblock {Studies in the theory of numbers}.
\newblock {The University of Chicago Science Series. Chicago: The University of
  Chicago Press. x, 230 p. (1930).}, 1930.

\bibitem[EG21]{EarnestGunawardana21JNT}
A.~G. Earnest and B.~L.~K. Gunawardana.
\newblock Local criteria for universal and primitively universal quadratic
  forms.
\newblock {\em J. Number Theory}, 225:260--280, 2021.

\bibitem[He22]{He22pre}
Z.~He.
\newblock On classic n-universal quadratic forms over dyadic local fields.
\newblock available at arXiv:2206.04885, 2022.

\bibitem[HH22]{HeHu22pre}
Z.~He and Y.~Hu.
\newblock On $n$-universal quadratic forms over dyadic local fields.
\newblock available at arXiv:2204.01997, 2022.

\bibitem[HHX22]{HeHuXu22}
Z.~He, Y.~Hu, and F.~Xu.
\newblock On indefinite $k$-universal quadratic forms over number fields.
\newblock arXiv:2201.10730, 2022.

\bibitem[HKK78]{HsiaKitaokaKneser78crelle}
J.~S. Hsia, Y.~Kitaoka, and M.~Kneser.
\newblock Representations of positive definite quadratic forms.
\newblock {\em J. Reine Angew. Math.}, 301:132--141, 1978.

\bibitem[Hsi75]{Hsia75Pacific}
J.~S. Hsia.
\newblock Spinor norms of local integral rotations. {I}.
\newblock {\em Pacific J. Math.}, 57(1):199--206, 1975.

\bibitem[HSX98]{HsiaShaoXu98crelle}
J.~S. Hsia, Y.~Y. Shao, and F.~Xu.
\newblock Representations of indefinite quadratic forms.
\newblock {\em J. Reine Angew. Math.}, 494:129--140, 1998.
\newblock Dedicated to Martin Kneser on the occasion of his 70th birthday.

\bibitem[Kim04]{Kim04ContempMath344}
M.-H. Kim.
\newblock Recent developments on universal forms.
\newblock In {\em Algebraic and arithmetic theory of quadratic forms}, volume
  344 of {\em Contemp. Math.}, pages 215--228. Amer. Math. Soc., Providence,
  RI, 2004.

\bibitem[KKO99]{KimKimOh99ContempMath249}
B.~M. Kim, M.-H. Kim, and B.-K. Oh.
\newblock {$2$}-universal positive definite integral quinary quadratic forms.
\newblock In {\em Integral quadratic forms and lattices ({S}eoul, 1998)},
  volume 249 of {\em Contemp. Math.}, pages 51--62. Amer. Math. Soc.,
  Providence, RI, 1999.

\bibitem[KKR97]{KimKimRaghavan97}
B.~M. Kim, M.-H. Kim, and S.~Raghavan.
\newblock {$2$}-universal positive definite integral quinary diagonal quadratic
  forms.
\newblock {\em Ramanujan J.}, 1(4):333--337, 1997.

\bibitem[{Ko}37]{Ko37}
C.~{Ko}.
\newblock {On the representation of a quadratic form as a sum of squares of
  linear forms}.
\newblock {\em {Q. J. Math., Oxf. Ser.}}, 8:81--98, 1937.

\bibitem[{Mor}30]{Mordell30}
L.~J. {Mordell}.
\newblock {A new Waring's problem with squares of linear forms}.
\newblock {\em {Q. J. Math., Oxf. Ser.}}, 1:276--288, 1930.

\bibitem[Oh00]{Oh00PAMS}
B.-K. Oh.
\newblock Universal {${\bf Z}$}-lattices of minimal rank.
\newblock {\em Proc. Amer. Math. Soc.}, 128(3):683--689, 2000.

\bibitem[O'M55]{OMeara55AJM}
O.~T. O'Meara.
\newblock Quadratic forms over local fields.
\newblock {\em Amer. J. Math.}, 77:87--116, 1955.

\bibitem[O'M58]{OMeara58AJM}
O.~T. O'Meara.
\newblock The integral representations of quadratic forms over local fields.
\newblock {\em Amer. J. Math.}, 80:843--878, 1958.

\bibitem[O'M00]{OMeara00}
O.~T. O'Meara.
\newblock {\em Introduction to quadratic forms}.
\newblock Classics in Mathematics. Springer-Verlag, Berlin, 2000.
\newblock Reprint of the 1973 edition.

\bibitem[{Ram}17]{Ramanujan17}
S.~{Ramanujan}.
\newblock {On the expression of a number in the form \(ax^2+by^2+cz^2+du^2\)}.
\newblock {\em {Proc. Camb. Philos. Soc.}}, 19:11--21, 1917.

\bibitem[Rie64]{Riehm64AJM}
C.~Riehm.
\newblock On the integral representations of quadratic forms over local fields.
\newblock {\em Amer. J. Math.}, 86:25--62, 1964.

\bibitem[{Ros}32]{Ross32}
A.~E. {Ross}.
\newblock {On representation of integers by quadratic forms}.
\newblock {\em {Proc. Natl. Acad. Sci. USA}}, 18:600--608, 1932.

\bibitem[Ros33]{Ross33}
A.~E. Ross.
\newblock On {R}epresentation of {I}ntegers by {I}ndefinite {T}ernary
  {Q}uadratic {F}orms of {Q}uadratfrei {F}orms of {Q}uadratfrei {D}eterminant.
\newblock {\em Amer. J. Math.}, 55(1-4):293--302, 1933.

\bibitem[XZ22]{XuZhang22TAMS}
F.~Xu and Y.~Zhang.
\newblock On indefinite and potentially universal quadratic forms over number
  fields.
\newblock {\em Trans. Amer. Math. Soc.}, 375(4):2459--2480, 2022.

\end{thebibliography}



Contact information of the authors:

\

Zilong HE

\medskip

Department of Mathematics

Southern University of Science and Technology


Shenzhen 518055, China


Email: hezl6@sustech.edu.cn

\

Yong HU

\medskip

Department of Mathematics

Southern University of Science and Technology


Shenzhen 518055, China


Email: huy@sustech.edu.cn

\

\end{document}